\documentclass[preprint,12pt,3p]{elsarticle}
\usepackage{amsmath,amsthm,amsfonts,amssymb}
\usepackage{fullpage}
\usepackage{subfig}
\usepackage{blkarray}
\usepackage{caption}
\usepackage{float}
\usepackage{algpseudocode}
\usepackage{tikz}
\usepackage{hyperref}
\numberwithin{equation}{section}
\usepackage{xcolor}
\usepackage{colortbl}
\usepackage[normalem]{ulem}
\usepackage{sectsty}
\usepackage{calrsfs}

\subsectionfont{\color{astral}}
\sectionfont{\color{astral}}
\usepackage{colortbl}

\DeclareFontFamily{OT1}{pzc}{}
\DeclareFontShape{OT1}{pzc}{m}{it}{<-> s * [0.900] pzcmi7t}{}

\usepackage{amsbsy}
\usepackage{amsmath}
\usepackage{accents}
\newlength{\dhatheight}

\newenvironment{frcseries}{\fontfamily{frc}\selectfont}{}

\definecolor{astral}{RGB}{46,116,181}
\subsectionfont{\color{astral}}
\sectionfont{\color{astral}}
\usepackage{colortbl}

\definecolor{darkslategray}{rgb}{0.18, 0.31, 0.31}
\definecolor{warmblack}{rgb}{0.0, 0.26, 0.26}

\definecolor{darkslategray}{rgb}{0.18, 0.31, 0.31}
\definecolor{warmblack}{rgb}{0.0, 0.26, 0.26}

\usepackage{multirow}
\usepackage{mathtools}
\usepackage{algorithm}
\usepackage{algorithmicx}
\makeatletter
\def\bState{\State\hskip-\aLG@thistlm}
\makeatother

\newtheorem{theorem}{Theorem}[section]
\newtheorem{lemma}[theorem]{Lemma}
\newtheorem{corollary}[theorem]{Corollary}
\theoremstyle{definition}
\usepackage{hyperref}
\newtheorem{definition}{Definition}[section]
\newtheorem{remark}{Remark}[section]
\newtheorem{example}{Example}[section]
\journal{.....}

\begin{document}

\begin{frontmatter}

\title{On  generalized-Drazin inverses and  GD-star matrices}

\vspace{-.4cm}

\author{Amit Kumar$^{\mu, a}$, Vaibhav Shekhar$^{\dagger, b}$ and Debasisha Mishra$^{*,a}$}

 \address{
                          $^a$Department of Mathematics,\\
                        National Institute of Technology Raipur\\ 
                        Chhattishgrah, India.\\
                        $^b$Department of Mathematics,\\ Indian Institute of Technology Delhi,
                       India.
                        \\email$^\mu$: amitdhull513@gmail.com  \\
                        email$^{\dagger}$: vaibhav@maths.iitd.ac.in\\ 
                        email$^*$: kapamath\symbol{'100}gmail.com. \\}
\vspace{-2cm}

\begin{abstract} 
Motivated by the works of Wang and Liu [Linear Algebra Appl.,
488 (2016) 235-248; MR3419784] and Mosi\'c [Results Math., 75(2) (2020) 1-21; MR4079761], we provide further results on GD inverses and
 introduce two new classes for square matrices called  GD-star (generalized-Drazin-star) and GD-star-one  (generalized-Drazin-star-one) using a GD inverse of a matrix.  We then exploit their various properties and characterize them in terms of various generalized inverses. We establish a representation of a GD-star matrix by using the core-nilpotent decomposition and Hartwig-Spindelb$\ddot{\text{o}}$ck
decomposition. We also define a binary relation called GD-star order using this class of matrices. Further, we obtain some analogous results for the class of star-GD  matrices. Moreover, the reverse-order law and forward-order law for GD inverse along with its monotonicity criteria are obtained.





 
 \end{abstract}


\end{frontmatter}

\newpage
\section{Introduction and motivation}
The literature for generalized inverses is quite large due to their huge  applicability in several fields. These inverses are applied to solve problems that appear in numerical analysis, statistics, neural computing, chemical equations, coding theory, robotics, etc. Interested readers are referred to \cite{baks}, \cite{ben11}, \cite{CD},    {\cite{app2}}, {\cite{app4}}, {\cite{app1}}, \cite{Mitra}, \cite{Mosic4}, and {\cite{app3}} for several applications of generalized inverses of matrices. 


Throughout this article, we denote $\mathbb{C}^{m\times n}$ to represent the set of all $m\times n$ complex matrices. For a given $A\in\mathbb{C}^{m\times n}$, the notions  $A^*$, $R(A)$, $N(A)$, and $P_A$ denote the conjugate transpose of $A$, the range space of $A$, the null space of $A$, and the orthogonal projection onto the range space of $A$, respectively. For every $A\in \mathbb{C}^{m\times n}$, the unique matrix $X\in \mathbb{C}^{n\times m}$ that satisfies the following four matrix equations: 
$$(1) AXA=A, ~(2) XAX=X, ~(3) (AX)^*=AX,~ \textnormal{ and ~} (4) (XA)^*=XA$$
is called the {\it Moore-Penrose inverse} \cite{Penrose} of $A$. It is denoted as $A^{\dagger}$. The Moore-Penrose inverse is applicable in finding the least-squares solution of minimum norm of an inconsistent linear system \cite{ben11}. The set of all matrices $X\in\mathbb{C}^{n\times m}$ which satisfies any of the combinations of the above four matrix equations is denoted as $A\{i,j,k,l\}$, where $i,j,k,l \in \{1,2,3,4\}$. For example, $A\{1\}$ denotes the set of all solutions $X$ of matrix equation (1). Such an $X$ which satisfies equation (1) is called first inverse or inner inverse of $A$, and is denoted by $A^{-}$ or $A^{(1)}$. Similarly, $A\{1,3\}$ denotes the set of all solutions of the first and third matrix equations. We denote a member of $A\{1,3\}$ as $A^{(1,3)}$. For $A\in \mathbb{C}^{m\times m}$, the smallest nonnegative integer for which $rank(A^k)=rank(A^{k+1})$ is called the {\it index} of the matrix $A$, and we denote it by $ind(A)$. Let $A\in\mathbb{C}^{m\times m}$, the unique matrix $X\in\mathbb{C}^{m\times m}$ satisfying the  equations $XAX=X$, $XA=AX$ and $A^{k+1}X=A^k$ is called the {\it Drazin inverse} \cite{ben11} of the matrix $A$. It is denoted as $A^D$. Here, $k$ denotes the index of the matrix $A$. The Drazin inverse is used to find solutions of singular differential equations \cite{ben11}. If $ind(A)=1$, then the unique $X$ satisfying the same matrix equations is called the {\it group inverse}.  More precisely, for $A\in\mathbb{C}^{m\times m}$ of index $k=1$, there exists a unique matrix $X$, called the {\it group inverse} of $A$ that satisfies   $XAX=X$, $XA=AX$ and $A^2X=A$. The group inverse of a matrix $A$ is denoted as $A^\#$. 
The group inverse helps to solve a statistic problem involving Markov chains \cite{ben11} (for example stationary probabilities, etc.).
In 1978, Campbell and Meyer \cite{sl1} provided some modifications to the classic Drazin inverse
by introducing weak Drazin inverse. For an instance the author showed that if one has a block triangular 
matrix it is easier to compute a weak Drazin than the Drazin inverse. The next result is an application of a weak Drazin inverse.
\begin{theorem}(\cite{sl1})
    If $T$ is the transition matrix of an $m$-state ergodic chain 
and if $A = I - T$, then the rows of $I - A^{WD}A$ are all equal to the unique fixed 
probability vector $w^*$ of $T$ for any generalized Drazin inverse of $A$, where $A^{WD}$ denotes a weak Drazin inverse of $A$.
\end{theorem}
\noindent Thereafter, several other generalized inverses have been introduced, namely, Bott-Duffin inverse \cite{ben11}, the core inverse \cite{baks} and MPCEP inverse \cite{Mosic4}, etc.

We next recall the definition of a particular type of weak Drazin inverse called generalized Drazin (or GD) inverse introduced by Wang and Liu \cite{Wang2} in 2016.
 For $A \in \mathbb{C}^{m\times m}$, a matrix $X \in \mathbb{C}^{m\times m}$ satisfying the following matrix equations  
$$AXA = A,~(6)XA^{k+1} = A^{k}  ~\text{and } (7)A^{k+1}X = A^{k}$$
is called {\it GD inverse} of $A$, where $ind(A)=k$. It is denoted by $X=A^{GD}$. 
Unlike the Drazin inverse, a GD inverse of a matrix need not be unique. The set of all GD inverses of a matrix $A$ is denoted by $A\{GD\}$. In 2020, Hern\'andez {\it et al.} \cite{thome3} introduced a new generalized inverse called {\it GDMP inverse} which is also not unique. The definition of a GDMP inverse is stated next.
Let $A \in\mathbb{ C}^{m\times m}$ and $ind(A)=k$. For each $A^{GD} \in A\{GD\}$, a GDMP inverse of $A$, denoted by $A^{GD\dagger}$, is an $m \times m$ matrix $A^{GD\dagger} = A^{GD}AA^{\dagger}$.  The symbol $A\{GD\dagger\}$ stands for the set of all GDMP inverses of $A$.\\  

For two  invertible matrices $A$ and $B$, the inverse of the product  $(AB)^{-1}=B^{-1}A^{-1}$ and $(AB)^{-1}=A^{-1}B^{-1}$ are known as reverse-order law and forward-order law, respectively. For invertible matrices, the reverse-order law always holds, while the forward-order law does not hold always. Further, these laws do not hold for generalized inverses in general. In the theory of generalized inverses, one of the fundamental topic of interest is to investigate reverse-order laws, forward-order laws, and additive properties, etc. Firstly, in 1966, Greville \cite{Grev} provided a few sufficient conditions so that the reverse-order law holds in case of the Moore-Penrose inverse, i.e., $(AB)^{\dagger}=B^{\dagger}A^{\dagger}$.
In the last few years, several authors also discussed the same problem for different generalized inverses. For instance, Xiong and Zheng \cite{Xiong16},  and Liu and Xiong \cite{Xiong25} presented reverse-order laws and forward-order laws, respectively, for  \{1,2,3\}- and \{1,2,4\}-inverses.
In 2016, Wang {\it et al.} \cite{Wang23} provided a few results of the reverse-order law for the Drazin inverse. Deng \cite{Deng24} studied the reverse-order law for the group inverse on Hilbert space.  In 2018, Castro-Gonz\'alez and Hartwig \cite{Nieves} provided some results on the forward-order law for the Moore-Penrose inverse, i.e., $(AB)^{\dagger}=A^{\dagger}B^{\dagger}.$ Very recently, Kumar {\it et al.} \cite{k11} presented  certain sufficient conditions for the reverse-order law and the forward-order law  for GD inverse and GDMP inverse. The reverse-order law has applications in the Karmarkar algorithm, and also used to analyze  Markov chains (see \cite{CDMeyer, Nieves}). For several applications of the forward-order law of generalized inverses in numerical linear algebra, one is referred to \cite{Dilan}, \cite{Guo} and \cite{Xion}.\\

Generalized inverses play an important role in the study of partial orders and pre-orders of matrices. These orders are crucial in the study of shorted operators which have their origin in the study of impedance matrices of $n$-port electrical networks (see \cite{Mitra} for more details). The definition of partial order  is recalled next. A binary relation $x$ (say) is called {\it partial order} if it is reflexive, anti-symmetric and transitive. 
The partial order was initially proposed in ring theory \cite{rpar}. In matrix setting, if a matrix $A$ is below $B$ under partial order $x$, then it is denoted by $A\leq^{x}B$. If a binary relation is reflexive and transitive, then it called {\it pre-order}. In matrix setting, a pre-order coincides with a partial order when the matrix is nilpotent.
The partial order and pre-order theories are also applied in fuzzy set theory, for example,  characterising the class of continuous t-norms \cite{thes}. For different applications of partial orders in many areas, such as statistics, electrical networks,
etc. one is referred to \cite{app-1}, \cite{app-2}, \cite{Mitra} and \cite{app3}.
 Below we recall some of the significant matrix partial orders and pre-orders that have been introduced in the literature. 
 \begin{enumerate}[(a)]
     \item $A\leq^{-}B$, i.e., $AA^{-}=BA^{-}$ and $A^{-}A=A^{-}B$ is called the minus partial order \cite{m2}.
     \item $A\leq^{\#}B$, i.e., $AA^{\#}=BA^{\#}$ and $A^{\#}A=A^{\#}B$ is called the group partial order \cite{m3}.
     \item $A\leq^{*}B$, i.e., $AA^{*}=BA^{*}$ and $A^{*}A=A^{*}B$ is called the  star partial order \cite{m1}.
     \item $A\leq^{D}B$, i.e., $AA^{D}=BA^{D}$ and $A^{D}A=A^{D}B$ is called the Drazin pre-order \cite{Mitra}.
 \end{enumerate}

 Recently, some new class of matrices have been introduced in the literature to solve certain matrix equations and several partial orders have been defined. In this direction, Mosi\'c  \cite{Mosic110} introduced a new class of matrix called the Drazin-star matrix which is stated next. Let $A\in \mathbb{C}^{m\times m}$ of index $k$, the Drazin-star matrix is denoted as $A^{D,*}$, is an $m\times m$ matrix $A^{D,*}=A^DAA^{*}$. Motivated by this work, we introduce two new classes of a matrix, which is called generalized-Drazin-star (or GD-star) matrix and GD-star-one matrix. We then investigate their properties. We show that the proposed classes of matrices serve as a solution to certain matrix equations and to system of linear equations. Further, sufficient conditions are obtained under which these classes of matrices coincide with some well known generalized inverses and some well known class of matrices. Additionally, we introduce a binary operation based on the proposed class of matrix and obtain some of its characterizations.\\
  
The aim of this article is to propose two new classes of matrices and investigate their properties. The article is organized as follows. Section \ref{vi} is divided into two subsections. First we reconstruct the representation of a $A^{GD}$ of matrix $A$ using the core-nilpotent decomposition and also some results on extension of a GD inverse (i.e., on GDMP inverse).  We then obtain spectral results for GD and GDMP inverses. The second subsection, i.e.,  Subsection \ref{sec5} establishes the reverse and forward-order laws of a GD inverse. Section \ref{sec3} is further divided into four subsections for better presentation.  Subsection \ref{2.1} proposes a new class of matrices, we call it GD-star matrices.  After that,  we use the core nilpotent decomposition and Hartwig-Spindelb$\ddot{\text{o}}$ck decomposition to find different representations of a GD-star matrix. Using the results of Section \ref{vi}, we obtained some results for GD-star matrices. In Section \ref{sec4}, we introduce a binary operation called GD-star order, and discuss its characterizations. Subsection \ref{sec6} provides a few results on dual GD-star, which are analogous to  those established in  Subsection \ref{2.1}. Finally, in Section \ref{sec7}, we introduce another class of matrices known as GD $*_1$ matrices and study its properties. 

\section{Some results on GD and GDMP inverses}\label{vi}
This section is divided into two subsections. The first subsection discusses a representation of a GD inverse.  We then present some spectral results on GD and GDMP inverses. Along with that, we also establish monotonicity condition for GD inverses. In the second subsection, we discuss the reverse and forward-order laws for GD inverses.
\subsection{Further results on GD and GDMP inverses}


Let $A\in \mathbb{C}^{m\times m}$. Then, it is well known that $A$ can be  expressed unitarily similar to an block triangular matrix containing a nonsingular and a nilpotent component in the diagonal block. The nonsingular component is called {\it core} and therefore, this decomposition is called as the {\it core-nilpotent decomposition}. The core-nilpotent decomposition of $A$ is given by
\begin{align}\label{ka5}
    A=P\begin{bmatrix}
{C}&S\\O&N\end{bmatrix}P^*,
\end{align}
where $P$ is a unitary matrix and $S$ is some matrix of suitable size.  The core-nilpotent decomposition is used to define the partial order (for example, minus, group and Drazin order) of the matrices (see \cite{Mitra}). Now, we reconstruct the representation of a $A^{GD}$ of matrix $A$ using \eqref{ka5}. The next result is in this direction.
\begin{theorem}
Let $A\in\mathbb{C}^{m\times m}$ be in the form of \eqref{ka5} with $ind(A)=k$. Then,
\begin{align}\label{ka7}
    A^{GD}=P\begin{bmatrix}
C^{-1}&C^{-(k+1)}(\widehat{T'}-\widehat{T}N^{-})\\O&N^{-}\end{bmatrix}P^*,
\end{align}
where $P$ is a unitary matrix and $\widehat{T'}-\widehat{T}N^{-}=-C^{k}SN^{-}+\displaystyle\sum_{j=0}^{k-1}C^{j}SN^{k-j-1}(I-NN^{-})$.
\end{theorem}

In next result, we demonstrate $A^{GD}AA^D=A^D$.
\begin{theorem}\label{dsa2}
Let $A\in \mathbb{C}^{m\times m}$ and $ind(A)=k$. Then, the system 
$$XA=AX,~XAX=X, \text{ and }A^{k+1}X=A^k$$
has a solution of the form $X=A^{GD}AA^{D}$, i.e., $A^{GD}AA^{D}=A^D$. 
\end{theorem}
\begin{proof}
Suppose $X=A^{GD}AA^{D}$. Then, $AX=AA^{GD}AA^D=AA^D$. Further,
\begin{align*}
    XA&=A^{GD}AA^DA\\
    &=A^{GD}A^{k+1}(A^D)^{k+1}A\\
    &=A^k(A^D)^{k+1}A&\\
    &=A^{k+1}(A^D)^{k+1}\\
    &=AA^D.
\end{align*}
So, $AX=XA$. Moreover,
\begin{align*}
    XAX&=A^{GD}AA^DAA^{GD}AA^D\\
    &=A^{GD}AA^DAA^{D}\\
    &=A^{GD}AA^D\\
    &=X,
\end{align*}
and
\begin{align*}
    A^{k+1}X&=A^{k+1}A^{GD}AA^D\\
    &=A^kAA^{D}\\
    &=A^{k+1}A^D\\
    &=A^k.
\end{align*}
Hence, $X=A^{GD}AA^D$ is Drazin inverse of $A$. By the uniqueness of Drazin inverse, we get $X=A^{GD}AA^D=A^D$.   
\end{proof}
Similarly, $A^{D}AA^{GD}$ is also the Drazin inverse. Hern\'andez {\it et al.} \cite{thome3} proved that  a GDMP  inverse  is a solution of the system $XAX=X$, $AX=AA^{\dagger}$ and $XA^{k}=A^{GD}A^k$. Motivated by this result, we will show that a GDMP inverse is $\{1,2,3, 6\}$-inverse of $A$.  
\begin{theorem}\label{dsa2}
Let $A\in \mathbb{C}^{m\times m}$ and $ind(A)=k$. Then, the system 
$$(AX)^*=AX,~AXA=A,~XAX=X, \text{ and }XA^{k+1}=A^k$$
has a solution of the form $X=A^{GD}AA^{\dagger}$ (GDMP). 
\end{theorem}
\begin{proof}
Set $X=A^{GD}AA^{\dagger}$. Now, $AX=AA^{GD}AA^{\dagger}=AA^{\dagger}$. So, $(AX)^*=AX$. Further,
\begin{align*}
    AXA&=AA^{GD}AA^{\dagger}A\\
    &=AA^{\dagger}A\\
    &=A,
\end{align*}
\begin{align*}
    XAX&=A^{GD}AA^{\dagger}AA^{GD}AA^{\dagger}\\
    &=A^{GD}AA^{GD}AA^{\dagger}\\
    &=A^{GD}AA^{\dagger}\\
    &=X,
\end{align*}
and 
\begin{align*}
    XA^{k+1}&=A^{GD}AA^{\dagger}A^{k+1}\\
    &=A^{GD}AA^k\\
    &=A^{GD}A^{k+1}\\
    &=A^k.
\end{align*}
\end{proof}

Next result can be proved in similar steps as the above theorem.
\begin{theorem}\label{dsa2}
Let $A\in \mathbb{C}^{m\times m}$ and $ind(A)=k$. Then, the system 
$$(XA)^*=XA,~AXA=A,~XAX=X, \text{ and }A^{k+1}X=A^k$$
has a solution of the form $X=A^{\dagger}AA^{GD}$ (MPGD). 
\end{theorem}
\begin{remark}
Since a GDMP inverse is also an $\{1,3\}$-inverse, so $A^{GD,\dagger}b$ give a least-squares solution of system $Ax=b$. Similarly, the minimum norm solution of $Ax=b$ is given by $x=A^{\dagger,GD}b$, where $A^{\dagger,GD}\in A\{\dagger,GD\}$.
\end{remark}
A matrix $A\in \mathbb{C}^{m\times m}$ is called {\it EP (or range-Hermitian}) if $R(A)=R(A^*)$. The class of EP matrices contain some special class of matrices including `Hermitian', `normal' and `nonsingular' matrices, to name few. An EP matrix appear  in the study of contraction \cite{Har2},  differential equation \cite{camb2}  and  difference equation \cite{camb1}, etc. The following is an important characterization of an EP matrix. A matrix $A$ is EP if and only if it commutes with its Moore-Penrose inverse, i.e., $AA^{\dagger}=A^{\dagger}A$.  Further, the Moore-Penrose, the group inverse and the Drazin inverse coincide for this class of matrices (see \cite{ben11}).
Suppose that $A=\begin{bmatrix}
    1&0\\0&0
\end{bmatrix}$ with index $k=1$. Since $A$ is EP, so, $A^{\dagger}=A^{\#}$ but there exist $A^{GD}=\begin{bmatrix}
    1&0\\0&a
\end{bmatrix}$ such that $A^{GD}\neq A^{\#}$ for $a\neq 0$.
Next, we establish the monotonicity  for GD inverse, but first we recall the following result given by Coll {\it et al.} \cite{thome1}.
\begin{corollary}(Corollary 2.1, \cite{thome1})\label{co2.1}
     Let $A, X \in\mathbb{ C}^{m\times m}$ and $ind(A)=k$. Then, the following conditions are equivalent:
     \begin{itemize}
         \item[(i)] $X \in A\{GD\}$.
\item[(ii)] $AXA = A$ and $A^kX = XA^k$.
\item[(iii)] $AXA = A$, $XA^{k+1} = A^k$, and $A^{k+1}X = A^k$.
   \end{itemize}
\end{corollary}

Mitra {\it et al.} \cite{Mitra} restated the core-nilpotent decomposition of matrix $A$ as follows:
\begin{align}\label{GD1}
    A=P\begin{bmatrix}
{C}&O\\O&N\end{bmatrix}P^{-1},
\end{align}
 where $P$ is nonsingular matrix. Wang and Liu \cite{Wang2} provided the representation of GD inverse of a matrix $A$ given by \eqref{GD1} which is stated next. Let $A\in\mathbb{C}^{m\times m}$. Then, $A^{GD}$ is a GD inverse of $A$ if and only if
\begin{align}\label{GD2}
    A^{GD}=P\begin{bmatrix}
{C}^{-1}&O\\O&N^{-}\end{bmatrix}P^{-1}.
\end{align}

Next, we recall a well known result by Stewart \cite{stewart}, which establishes a bound of norm of the matrix $(I_m+A)^{-1}$.
 \begin{theorem}(\cite{stewart})\label{ste}
 Let $A\in\mathbb{C}^{m\times m}$ with  $||A||\leq 1$. Then, $I_m+A$ is nonsingular and 
 $$||(I_m+A)^{-1}||\leq(I_m-||A||)^{-1}.$$
 \end{theorem}
 With the help of above theorem, we now give a perturbation result.
 \begin{theorem}
 Let $A\in\mathbb{C}^{m\times m}$ and $B=A+E\in\mathbb{C}^{m\times m}$ with $ind(A)=k$. If the perturbation $E$ satisfies $A^kE=E$ and $||A^{GD}E||<1$, then 
 $$(I_m+A^{GD}E)^{-1}A^{GD}\in B\{1\}.$$
 \end{theorem}
\begin{proof}
 Assume that $E=P\begin{bmatrix}E_1&E_2\\E_3&E_4
 \end{bmatrix}P^{-1}$ and $rank(A)=r$. From \eqref{GD1} and \eqref{GD2}, we have  $A=P\begin{bmatrix}C&O\\O&N
 \end{bmatrix}P^{-1}$ and $A^{GD}=P\begin{bmatrix}C^{-1}&O\\O&N^{-}
 \end{bmatrix}P^{-1}$, respectively. Now, $A^{k}=P\begin{bmatrix}C^k&O\\O&N^k\end{bmatrix}P^{-1}=P\begin{bmatrix}C^k&O\\O&O\end{bmatrix}P^{-1}$. Then,
 \begin{align*}
     A^kE=P\begin{bmatrix}C^k&O\\O&O\end{bmatrix}P^{-1}P\begin{bmatrix}E_1&E_2\\E_3&E_4
 \end{bmatrix}P^{-1}=P\begin{bmatrix}
     C^kE_1&C^kE_2\\O&O\end{bmatrix}P^{-1}.
 \end{align*}
 If $A^kE=E$, then we have $E_3=O$ and  $E_4=O$. Further, 
 \begin{align*}
     I_m+A^{GD}E&=I_m+P\begin{bmatrix}C^{-1}&O\\O&N^{-}\end{bmatrix}P^{-1}P\begin{bmatrix}E_1&E_2\\O&O\end{bmatrix}P^{-1}
     \end{align*}
     \begin{align*}
     &=I_m+P\begin{bmatrix}C^{-1}E_1&C^{-1}E_2\\O&O\end{bmatrix}P^{-1}\\
     &=P\begin{bmatrix}I_r+C^{-1}E_1&C^{-1}E_2\\O&I_{m-r}\end{bmatrix}P^{-1}\\
     &=P\begin{bmatrix}C^{-1}(C+E_1)&C^{-1}E_2\\O&I_{m-r}\end{bmatrix}P^{-1}. 
 \end{align*}
 From Theorem \ref{ste}, $||A^{GD}E||\leq1$ implies that $I_m+A^{GD}E$ is invertible. So,\\   $(I_m+A^{GD}E)^{-1}=P\begin{bmatrix}(C+E_1)^{-1}C&-(C+E_1)^{-1}E_2\\O&I_{m-r}\end{bmatrix}P^{-1}.$ Furthermore, $$(I_m+A^{GD}E)^{-1}A^{GD}=P\begin{bmatrix}(C+E_1)^{-1}&-(C+E_1)^{-1}E_2N^{-}\\O&N^{-}\end{bmatrix}P^{-1},$$  and $B=A+E=P\begin{bmatrix}C+E_1&E_2\\O&N\end{bmatrix}P^{-1}$. Now, we verify $(I_m+A^{GD}E)^{-1}A^{GD}\in B\{1\}$.  Clearly, 
 \begin{align*}
B(I_m+A^{GD}E)^{-1}A^{GD}B&=P\begin{bmatrix}C+E_1&E_2\\O&N\end{bmatrix}P^{-1}P\begin{bmatrix}(C+E_1)^{-1}&-(C+E_1)^{-1}E_2N^{-}\\O&N^{-}\end{bmatrix}\\
     &~~~~~~~~~~~~~~~~~~~~~~~~~~~~~~~~~~~~~~~~~P^{-1}P\begin{bmatrix}C+E_1&E_2\\O&N\end{bmatrix}P^{-1}\\
     &=P\begin{bmatrix}C+E_1&E_2\\O&N\end{bmatrix}\begin{bmatrix}I_r&(C+E_1)^{-1}E_2(I_{m-r}-N^{-}N)\\O&N^{-}N\end{bmatrix}P^{-1}\\
     &=P\begin{bmatrix}C+E_1&E_2(I_{m-r}-N^{-}N)+E_2N^{-}N\\O&NN^{-}N\end{bmatrix}P^{-1}\\
     &=P\begin{bmatrix}C+E_1&E_2\\O&N\end{bmatrix}P^{-1}=B.
 \end{align*}
 Hence, $(I_m+A^{GD}E)^{-1}A^{GD}\in B\{1\}$.
 \end{proof}
 \begin{remark}
 If in \eqref{GD2}, we replace $N^{-}$ by $N^{\dagger}$, then $B^{-} $ becomes $B^{\dagger}.$
 \end{remark}
\begin{theorem}\label{cd1}(\cite{Barata})
 Let $A\in \mathbb{C}^{m\times m}$ be a non-zero matrix and let $AA^*=\displaystyle\sum_{i=1}^r\alpha_iE_i$ be the spectral representation of $AA^*$, where $\{\alpha_1,\alpha_2,...,\alpha_r\}$ is the set of
distinct eigenvalues of $AA^*$ and $E_i$ are the corresponding self-adjoint spectral projections. Then, we have
$$A^{\dagger}=\displaystyle\sum_{i=1,~ \alpha_i\neq0}^r\dfrac{1}{\alpha_i}A^{*}E_i.$$
\end{theorem}
With the help of above theorem, we provide a spectrum result for GDMP inverse.
\begin{theorem}\label{spe}
 Let $A\in \mathbb{C}^{m\times m}$ be a non-zero matrix and let $AA^*=\displaystyle\sum_{i=1}^r\alpha_iE_i$ be the spectral representation of $AA^*$, where $\{\alpha_1,\alpha_2,...,\alpha_r\}$ is the set of
distinct eigenvalues of $AA^*$ and $E_i$ are the corresponding self-adjoint spectral projections. Then, we have
$$A^{GD,\dagger}=\displaystyle\sum_{i=1}^rA^{GD}E_i=A^{GD}.$$
\end{theorem}
    

\subsection{Reverse and Forward-order laws}\label{sec5}


 The reverse-order law and forward-order law are fundamental topics of generalized inverse. Kumar {\it et al.} \cite{k11} discussed the reverse-order law, forward-order law and additive properties of a GD inverse. In this section, we provide some more results for the reverse and forward-order laws of a GD inverse. We begin as a result of the reverse-order law of a GD inverse.
 \begin{theorem}\label{rev}
Let $A,B\in\mathbb{C}^{m \times m}$ be such that $AB=BA$ and $max\{ind(A),ind(B)\}=k$. If $R(A^{GD})\subseteq R(B)$, then $(AB)^{GD}=B^{GD}A^{GD}$.
\end{theorem}
\begin{proof}
The hypothesise $R(A^{GD})\subseteq R(B)$ implies that $A^{GD}=BY$, for some matrix $Y$. Pre-multiplying by $BB^{GD}$ in $A^{GD}=BY$, we get $BB^{GD}A^{GD}=A^{GD}$.
Set $X=B^{GD}A^{GD}$. Now,
\begin{align}
    ABXAB&=ABB^{GD}A^{GD}AB\nonumber\\
    &=AA^{GD}AB\text{ (since $BB^{GD}A^{GD}=A^{GD}$)}\nonumber\\
    &=AB,
\end{align}
and
\begin{align}
    X(AB)^{k+1}&=B^{GD}A^{GD}A^{k+1}B^{k+1}\text{ (since $AB=BA)$}\nonumber\\
    &=B^{GD}A^kB^{k+1}\nonumber\\
    &=B^{GD}B^{k+1}A^{k}\nonumber\\
    &=B^kA^k\nonumber\\
    &=(AB)^k.
\end{align}
Similarly, $(AB)^{k+1}X=(AB)^k$. Hence, $(AB)^{GD}=B^{GD}A^{GD}$. 
\end{proof}
 Next result can be proved as similar step to the above theorem.
\begin{theorem}\label{for}
Let $A,B\in\mathbb{C}^{m \times m}$ be such that $AB=BA$ and $max\{ind(A),ind(B)\}=k$. If $R(B^{GD})\subseteq R(A)$, then $(AB)^{GD}=A^{GD}B^{GD}$.
\end{theorem}
Some sufficient conditions for the triple reverse-order law of a GD inverse is listed below.
\begin{theorem}
Let $A,B,C\in\mathbb{C}^{m \times m}$ be such that commute with each others and\\ $max\{ind(A),ind(B), ind(C)\}=k$.  If one of the following holds:
\begin{enumerate}[(i)]
    \item $R(B^{GD}B)\subseteq R(C)$ and $A^{GD}BC=BCA^{GD}$,
    \item $R(A^{GD}A)\subseteq R(B)$ and $C^{GD}AB=ABC^{GD}$,
    \item $C^{GD}AB=ABC^{GD}$ and $A^{GD}AB=BA^{GD}A$,
    \item $A^{GD}BC=BCA^{GD}$ and $CC^{GD}B=BCC^{GD}$,
\end{enumerate}
then $(ABC)^{GD}=C^{GD}B^{GD}A^{GD}$.
\end{theorem}
\begin{proof} For the proof of this result, set $X=C^{GD}B^{GD}A^{GD}$. We will prove $X$ is a GD inverse of $ABC$  by using the definition of a GD inverse.\\ 
(i)  $R(B^{GD}B)\subseteq R(C)$ implies that  $B^{GD}B=CY$. Pre-multiplying by $CC^{GD}$ in $B^{GD}B=CY$, we get $CC^{GD}B^{GD}B=B^{GD}B$. Now,
\begin{align*}
    ABCXABC&=ABCC^{GD}B^{GD}A^{GD}ABC\\
    &=ABCC^{GD}B^{GD}BCA^{GD}A \textnormal{ (since $A^{GD}BC=BCA^{GD}$)}\\
    &=ABB^{GD}BCA^{GD}A\\
    &=ABCA^{GD}A\\
    &=BCAA^{GD}A\\
    &=BCA\\
    &=ABC,
\end{align*}
\begin{align*}
    X(ABC)^{k+1}&=C^{GD}B^{GD}A^{GD}(ABC)^{k+1}\\
    &=C^{GD}B^{GD}A^{GD}A^{k+1}B^{k+1}C^{k+1}\\
    &=C^{GD}B^{GD}A^kB^{k+1}C^{k+1}\\
    &=C^{GD}B^{GD}B^{k+1}C^{k+1}A^k\\
    &=C^{GD}C^{k+1}B^kA^k\\
    &=C^kB^kA^k\\
    &=(ABC)^k,
\end{align*}
and
  \begin{align}\label{e5.6}
    (ABC)^{k+1}X&=(ABC)^{k+1}A^{GD}B^{GD}C^{GD}\nonumber\\
    &=A^{k+1}B^{k+1}C^{k+1}C^{GD}B^{GD}A^{GD}\nonumber\\
    &=A^{k+1}B^{k+1}C^kB^{GD}A^{GD}\nonumber\\
    &=C^kA^{k+1}B^{k+1}B^{GD}A^{GD}\nonumber\\
    &=C^kB^kA^{k+1}A^{GD}\nonumber\\
    &=C^kB^kA^k\nonumber\\
    &=(ABC)^k.
\end{align}
Hence, $(ABC)^{GD}=C^{GD}B^{GD}A^{GD}$.\\
(ii)  $X(ABC)^{k+1}= (ABC)^{k}$ and $(ABC)^{k+1}X= (ABC)^{k}$ both conditions proofs are similar to part (i). $R(A^{GD}A)\subseteq R(B)$ implies that $A^{GD}A=BZ$. Pre-multiplying by $BB^{GD}$ in $A^{GD}A=BZ$, we get $BB^{GD}A^{GD}A=A^{GD}A$.
\begin{align*}
    ABCXABC&= ABCC^{GD}B^{GD}A^{GD}ABC\\
    &=CC^{GD}ABB^{GD}A^{GD}ABC \text{ (since $C^{GD}AB=ABC^{GD}$) }\\
    &=CC^{GD}AA^{GD}ABC\\
    &=CC^{GD}ABC\\
    &=CC^{GD}CAB\\
    &=CAB\\
    &=ABC.
\end{align*}

(iii)  $X(ABC)^{k+1}= (ABC)^{k}$ and $(ABC)^{k+1}X= (ABC)^{k}$ both conditions proofs are similar to part (i). Further, \begin{align*}
    ABCXABC&= ABCC^{GD}B^{GD}A^{GD}ABC\\
    &=CC^{GD}ABB^{GD}A^{GD}ABC \textnormal{ (since $C^{GD}AB=ABC^{GD}$)}\\
    &=CC^{GD}ABB^{GD}BA^{GD}AC\textnormal{ (since $A^{GD}AB=BA^{GD}A$)}\\
    &=CC^{GD}BAA^{GD}ABC\\
    &=CC^{GD}ABC\\
    &=CC^{GD}CAB\\
    &=CAB\\
    &=ABC.
\end{align*}
(iv) The proofs for  $X(ABC)^{k+1}= (ABC)^{k}$ and $(ABC)^{k+1}X= (ABC)^{k}$ are similar to part (i). Furthermore,  \begin{align*}
    ABCXABC&=ABCC^{GD}B^{GD}A^{GD}ABC\\
    &=ABCC^{GD}B^{GD}BCA^{GD}A \textnormal{ (since $A^{GD}BC=BCA^{GD}$)}\\
      &=ACC^{GD}BB^{GD}BCA^{GD}A \textnormal{ (since $CC^{GD}B=BCC^{GD}$)}\\
    &=ACC^{GD}CBA^{GD}A\\
    &=ABCA^{GD}A\\
    &=BCAA^{GD}A\\
    &=BCA\\
    &=ABC.
\end{align*}
\end{proof}
The proof of the triple forward-order law is similar as the proof of the triple reverse-order law of a GD inverse which is stated below skipping the proof.
\begin{theorem}
Let $A,B,C\in\mathbb{C}^{m \times m}$ be such that commute with each others and\\ $max\{ind(A),ind(B), ind(C)\}$=k. If one of the following holds:
\begin{enumerate}[(i)]
    \item $R(C^{GD}C)\subseteq R(B)$ and $A^{GD}BC=BCA^{GD}$,
    \item $R(B^{GD}B)\subseteq R(A)$ and $C^{GD}AB=ABC^{GD}$,
    \item $C^{GD}AB=ABC^{GD}$ and $A^{GD}AB=BA^{GD}A$,
    \item $A^{GD}BC=BCA^{GD}$ and $CC^{GD}B=BCC^{GD}$,
\end{enumerate}
then $(ABC)^{GD}=A^{GD}B^{GD}C^{GD}$.
\end{theorem}

\section{GD-star matrices and GD-star order}\label{sec3}

\noindent In this section, we discuss some of the main results of this article. In particular, we first propose a new class of matrices and then investigate its various properties. We also introduce a binary relation based on this class of matrix. In addition to these, we also present a new class of matrices, we  call it GD-star-one.\par

\subsection{GD-star matrices}\label{2.1}
 We start this section by defining a GD-star matrix.
\begin{definition}\label{sa1}
Let $A \in \mathbb{C}^{m\times m}$ and $ind(A)=k$. For each  $A^{GD} \in A\{GD\}$, a GD-star matrix of $A$, denoted by $A^{GD, *}$, is an $m \times m$ matrix
$$A^{GD, *} = A^{GD}AA^*.$$
\end{definition}

The above definition is motivated by the definition of  Drazin-star matrix and GDMP generalized inverse, and it may seem that the proposed class of matrices contained in the class of these matrices. To disprove this, we present two examples. The first example shows that GD-star matrix of a matrix is not unique. Also, this class is different from the well known  Drazin-star matrices. Similarly, the second example demonstrates that a GDMP inverse is different from a GD-star matrix. 
\begin{example}
For $A=\begin{bmatrix}
0&1\\0&0\end{bmatrix}$ with index $k=2$. We have $A^{GD}=\begin{bmatrix}
a&b\\1&c\end{bmatrix}$, where $a,b,c\in\mathbb{C}$. Further, $A^D=O$, $A^{GD,*}=\begin{bmatrix}a&0\\1&0\end{bmatrix}$, and $A^{D,*}=O$. So, $A^{GD}\neq A^{D}$ and $A^{GD,*}\neq A^{D,*}.$
\end{example}
\begin{example}\label{exam3}
Let $A=\begin{bmatrix}
1&1\\0&0\end{bmatrix}$ with index $k=1$. Then, $A^{\dagger}=\begin{bmatrix}0.5&0\\0.5&0\end{bmatrix}$ and $A^{GD}=\begin{bmatrix}
1&a\\0&b\end{bmatrix}\in A\{GD\}$, where $a+b=1$. Now,    $A^{GD\dagger}=A^{GD}AA^{\dagger}=\begin{bmatrix}1&0\\0&0\end{bmatrix}$ and $A^{GD,*}=A^{GD}AA^*=\begin{bmatrix}2&0\\0&0\end{bmatrix}$, i.e., $A^{GD\dagger}\neq A^{GD,*}$. It is clear that the class of GD-star matrix is different from the class of GDMP inverse.
\end{example}
\begin{remark}
Since
$A^{GD}AA^*=A^{GD}A(AA^{\dagger}A)^*
    =A^{GD}AA^{\dagger}AA^*
    =A^{GD,\dagger}AA^*,$ the Definition \ref{sa1} remains unaffected if we replace a GD inverse by a GDMP inverse of $A$, i.e., if we replace $A^{GD}$ by $A^{GD,\dagger}$.
Similarly, one may observe that the definition of Drazin-star matrix of $A$, i.e., $A^{D,*}$  remains unchanged if we replace $A^D$ by $A^{D,\dagger}$.
\end{remark}
Now, we prove that a GD-star matrix is a solution of the matrix equations $X(A^{\dagger})^*X=X,~A^kX=A^kA^*, \text{ and } X(A^{\dagger})^*=A^{GD}A$. 
\begin{theorem}\label{sa2}
Let $A\in \mathbb{C}^{m\times m}$. Then, for $k\geq 1$, the system 
$$X(A^{\dagger})^*X=X,~A^kX=A^kA^*, \text{ and } X(A^{\dagger})^*=A^{GD}A$$
has a solution of the form $X\in A\{GD,*\}$, where $A\{GD,*\}$ is the set of all GD-star matrices of $A$. 
\end{theorem}

\begin{proof}
We will show that $X\in A\{GD,*\}$ satisfies the given matrix equations. First we see that
\begin{align}\label{ka1}
    X(A^{\dagger})^*X&=A^{GD}AA^*(A^{\dagger})^*A^{GD}AA^*\nonumber\\
    &=A^{GD}A(A^{\dagger}A)^*A^{GD}AA^*\nonumber\\
    &=A^{GD}AA^{\dagger}AA^{GD}AA^*\nonumber\\
    &=A^{GD}AA^{GD}AA^*\nonumber\\
    &=A^{GD}AA^{*}\nonumber\\
    &=X.
\end{align}


\noindent When $k\geq 1$, 
\begin{align}\label{ka2}
    A^kX&=A^kA^{GD}AA^*\nonumber\\
    &=A^{k-1}AA^{GD}AA^*\nonumber\\
    &=A^{k-1}AA^*\nonumber\\
    &=A^kA^*,
\end{align}
and
\begin{align}\label{ka3}
    X(A^{\dagger})^*&=A^{GD}AA^*(A^{\dagger})^*\nonumber\\
    &=A^{GD}A(A^{\dagger}A)^*\nonumber\\
    &=A^{GD}AA^{\dagger}A\nonumber\\
    &=A^{GD}A.
\end{align}
Thus, from \eqref{ka1}, \eqref{ka2} and \eqref{ka3} it is clear that $X\in A\{GD,*\}$ satisfies the given matrix equations.
\end{proof}
Next we state a few properties that a GD-star matrix of a matrix $A$ possesses.
\begin{lemma}\label{sa3}
Let $A \in \mathbb{C}^{m\times m}$, and $ind(A)=k$. If $A^{GD}\in A\{GD\}$, then a GD-star matrix $X$ of the matrix $A$ satisfies the following properties:
\begin{enumerate}[(i)]
   
     \item $AX(A^{\dagger})^*=A$.
    \item $A^kX=A^{k+1}A^{GD}A^*=A^{GD}A^{k+1}A^*$.
    \item $A^{\dagger}AX=A^*$.
    \item $A^kX(A^{\dagger})^*=A^k$.
    \item $X(A^{\dagger})^*A^{GD}=A^{GD}AA^{GD}$.  
    \item $X(A^{\dagger})^*A^k=A^k$.
    \item $A^{\dagger}AX^2=A^*X$.
    \item $A^{\dagger}AX^2AA^{\dagger}=A^*X$.
    \item $XAX=A^{GD}(AA^*)^2$.
    \item $XAA^{\dagger}X=X^2$.
    \item $(AX)^*=AX$.
    \item $(A^{\dagger})^*X(A^{\dagger})^*=(A^{\dagger})^*.$
        \item $(A^{\dagger})^*X=AA^{\dagger}$.
    \item $(X(A^{\dagger})^*)^2=X(A^{\dagger})^*$.
     \item $X\in (A^{\dagger})^*\{2,3\}$.
\end{enumerate}
\end{lemma}
\begin{proof}
\begin{enumerate}[(i)]
\item We have $AX(A^{\dagger})^*=AA^{GD}AA^*(A^{\dagger})^*=A(A^{\dagger}A)^*=AA^{\dagger}A=A$. So, $AX(A^{\dagger})^*=A$.
\item Observe that $A^kX=A^kA^{GD}AA^*=A^{GD}A^{k+1}A^{GD}AA^*=A^{GD}A^kAA^*=A^{GD}A^{k+1}A^*=A^{k}A^*=A^{k+1}A^{GD}A^*$. Hence, $A^kX=A^{k+1}A^{GD}A^*=A^{GD}A^{k+1}A^*$.
\item Clearly, $A^{\dagger}AX=A^{\dagger}AA^{GD}AA^*=A^{\dagger}AA^*=A^*(A^*)^{\dagger}A^*=A^*$, i.e., $A^{\dagger}AX=A^*$.
\item From (ii), we have $A^kX=A^{k+1}A^{GD}A^*$, therefore, $A^kX(A^{\dagger})^*=A^{k+1}A^{GD}A^*(A^{\dagger})^*=A^k(A^{\dagger}A)^*=A^{k-1}AA^{\dagger}A=A^k.$
\item By Theorem \ref{sa2}, we have $X(A^{\dagger})^*=A^{GD}A$, and thus  $X(A^{\dagger})^*A^{GD}=A^{GD}AA^{GD}$.
\item Applying Theorem \ref{sa2} again, we obtain $X(A^{\dagger})^*A^k=A^{GD}AA^k=A^{GD}A^{k+1}=A^k$, i.e., $X(A^{\dagger})^*A^k=A^k$.
\item  This is obvious as $A^{\dagger}AX^2=A^{\dagger}AA^{GD}AA^*X=A^{\dagger}AA^*X=A^*X$.
\item From (vii), we get $A^{\dagger}AX^2=A^*X$. Now, $A^{\dagger}AX^2AA^{\dagger}=A^*XAA^{\dagger}=A^*A^{GD}AA^*(A^*)^{\dagger}A^*=A^*A^{GD}AA^*$. Hence, $A^{\dagger}AX^2AA^{\dagger}=A^*X$.
\item Clearly, $XAX=A^{GD}AA^*AA^{GD}AA^*=A^{GD}AA^*AA^*=A^{GD}(AA^*)^2$.
\item We have $XAA^{\dagger}X=A^{GD}AA^*AA^{\dagger}X=A^{GD}AA^*(A^*)^{\dagger}A^*X=A^{GD}AA^*X=XX=X^2$. Therefore, $XAA^{\dagger}X=X^2$.
\item This is quite trivial as $AX=AA^{GD}AA^*=AA^*$, i.e., $(AX)^*=AX$.
\item By Theorem \ref{sa2}, we have $X(A^{\dagger})^*=A^{GD}A$. Since, $(A^{\dagger})^*X(A^{\dagger})^*=(A^{\dagger})^*A^{GD}A=(A^{\dagger}AA^{\dagger})^*A^{GD}A=(A^{\dagger})^*A^{\dagger}AA^{GD}A=(A^{\dagger})^*A^{\dagger}A=(A^{\dagger})^*$. Hence, $(A^{\dagger})^*X(A^{\dagger})^*=(A^{\dagger})^*.$

 
\item
We have
\begin{align}\label{ka8}
    (A^{\dagger})^*X&=(A^{\dagger})^*A^{GD}AA^*\nonumber\\
    &=(A^{\dagger}AA^{\dagger})^*A^{GD}AA^*\nonumber\\
    &=(A^{\dagger})^*A^{\dagger}AA^{GD}AA^*\nonumber\\
    &=(A^{\dagger})^*A^{\dagger}AA^*\nonumber\\
    &=(A^{\dagger})^*A^*(A^*)^{\dagger}A^*\nonumber\\
    &=(A^{\dagger})^*A^*\nonumber\\
    &=AA^{\dagger}.
\end{align}
\item Clearly, $(X(A^{\dagger})^*)^2=X(A^{\dagger})^*X(A^{\dagger})^*=A^{GD}AA^*(A^{\dagger})^*A^{GD}AA^*(A^{\dagger})^*=$\\$A^{GD}AA^{GD}AA^*(A^{\dagger})^*=A^{GD}AA^*(A^{\dagger})^*=X(A^{\dagger})^*$. 
\item It can be easily obtained from \eqref{ka8} and part (xiii) of Lemma \ref{sa3}.

\end{enumerate}
\end{proof}

\begin{remark}
    From (xiii), we have $(A^{\dagger})^*X$ is an orthogonal projector onto  $R(A)$.
\end{remark}


Under some  conditions, a GD-star matrix coincides with some well known generalized inverses. The same result is proved next.

\begin{theorem}\label{thm3.2}
Let $A \in \mathbb{C}^{m\times m}$ and $ind(A)=k$. For each $A^{GD}\in A\{GD\}$, a GD-star matrix $X$ of the matrix $A$ satisfies the following properties:
\begin{enumerate}[(i)]
    \item If $A^*=A$, then $X\in A^{\dagger}\{1,2,3\}$.
\item If $A^{GD}=A^{\#}$, then $X=A^{\#,*}.$
\item If $A$ is a partial isometry, then $X=A^{GD,\dagger},$
and thus $AX=P_{R(A)}$.
\end{enumerate}
\end{theorem}
\begin{proof}
(i) Since, $A^{\dagger}X=(A^{\dagger})^*X=AA^{\dagger}$ by equation \eqref{ka8} and $A^*=A$. We have\begin{align*}
XA^{\dagger}X&=A^{GD}AA^*A^{\dagger}A^{GD}AA^*\\
    &=A^{GD}AAA^{\dagger}A^{GD}AA\\
    &=A^{GD}AA^{GD}AA \text{ ($A^*=A$ implies $AA^{\dagger}=A^{\dagger}A$)}\\
    &=A^{GD}AA\\
    &=A^{GD}AA^*\\
    &=X,
\end{align*}
and
\begin{align*}
    A^{\dagger}XA^{\dagger}&=(A^{\dagger})^*A^{GD}AA^*(A^{\dagger})^* \text{ (since $A^*=A$)}\\
    &=(A^{\dagger}AA^{\dagger})^*A^{GD}A\\
    &=(A^{\dagger})^*A^{\dagger}AA^{GD}A\\
    &=(A^{\dagger})^*A^{\dagger}A\\
    &=A^{\dagger}.
\end{align*}\\
(ii) We have $A^{GD}=A^{\#}$. So, $k=1$. Then, $A^{GD,*}=A^{GD}AA^*=A^{\#}AA^*=A^{\#,*}.$\\
(iii) Since $A$ is a partial isometry, therefore, $A^*=A^{\dagger}$. Thus, $A^{GD,*}=A^{GD}AA^*=A^{GD}AA^{\dagger}=A^{GD,\dagger}.$ Therefore, $AA^{GD,*}=AA^{\dagger}=P_{R(A)}$.
\end{proof}

Next we collect some properties of a GD-star matrix for the class of EP matrices. 
\begin{theorem}\label{thm32}
Let $A \in \mathbb{C}^{m\times m}$ be an EP matrix. For each $A^{GD}\in A\{GD\}$, a GD-star matrix $X$ of the matrix $A$ satisfies the following properties:
\begin{enumerate}[(i)]
    \item $A^{\dagger}X=A^{\dagger}A^*$.
\item $AA^{\dagger}XAA^{\dagger}=A^*$.
\item $X=A^*$.
\end{enumerate}
\end{theorem}
\begin{proof}
(i) We have\begin{align*}
    A^{\dagger}X&=A^{\dagger}A^{GD}AA^*\\
    &=A^{\dagger}AA^{\dagger}A^{GD}AA^*\\
    &=A^{\dagger}A^{\dagger}AA^{GD}AA^{*}\\
    &=A^{\dagger}A^{\dagger}AA^*\\
    &=A^{\dagger}AA^{\dagger}A^*\\
    &=A^{\dagger}A^*.
\end{align*}
(ii) Again, \begin{align*}
    AA^{\dagger}XAA^{\dagger}&=A^{\dagger}AA^{GD}AA^*AA^{\dagger}\\
    &=A^{\dagger}AA^*(A^*)^{\dagger}A^*\\
    &=A^{\dagger}AA^{*}\\
    &=A^*(A^*)^{\dagger}A^*\\
    &=A^*.
    \end{align*}
    (iii)  We know $ind(A)=1$ because $A$ is an EP matrix, so, $AA^{GD}=A^{GD}A$ by Corollary \ref{co2.1} (ii). Now, we have
    $X=A^{GD}AA^*=AA^{GD}(AA^{\dagger}A)^*    =AA^{GD}A^{\dagger}AA^*=AA^{GD}AA^{\dagger}A^*=AA^{\dagger}A^*=A^*.$ Therefore, if $A$ is an EP matrix, then $A^{GD,*}=A^*$.

\end{proof}

\begin{remark}
\begin{enumerate}

    \item If $A$ is an EP-Hermitian matrix, then $A^{GD,*}=A$.
    \item  Since $A$ is an EP matrix, so $AA^{\dagger}=A^{\dagger}A$. Again, using Corollary \ref{co2.1} (ii), we have $A^{GD}A=AA^{GD}$. Therefore, $A^{GD\dagger}=A^{GD}AA^{\dagger}=AA^{GD}A^{\dagger}AA^{\dagger}=AA^{GD}AA^{\dagger}A^{\dagger}=AA^{\dagger}A^{\dagger}=A^{\dagger}$. Similarly, one can show that $A^{D,*}=A^{\dagger}$. Hence, $A^{GD\dagger}=A^{\dagger}=A^{D,*}$ whenever $A$ is EP.
\end{enumerate}
\end{remark}

Now, we establish a representation of a GD-star matrix. 
\begin{theorem}\label{thm2.8}
Let $A\in\mathbb{C}^{m\times m}$ be in the form of \eqref{ka5} with $ind(A)=k$. If  $A^{GD}\in A\{GD\}$, then 
\begin{align}\label{ka6} A^{GD,*}=P\begin{bmatrix}C^*+C^{-1}SS^*+C^{-(k+1)}(\widehat{T'}-\widehat{T}N^{-})NS^*&C^{-1}SN^*+C^{-(k+1)}(\widehat{T'}-\widehat{T}N^{-})NN^*\\N^{-}NS^*&N^{-}NN^* \end{bmatrix}P^*,\end{align}
where $P$ is a unitary matrix and $\widehat{T'}-\widehat{T}N^{-}=-C^{k}SN^{-}+\displaystyle\sum_{j=0}^{k-1}C^{j}SN^{k-j-1}(I-NN^{-})$.
\end{theorem}

If $S=O$ in Theorem \ref{thm2.8}, then the 
above theorem reduces to the following corollary.
\begin{corollary}
Let $A\in\mathbb{C}^{m\times m}$. If $P$ is a unitary matrix in \eqref{GD1}, then 
$$A^{GD,*}=P\begin{bmatrix}C^*&O\\O&N^{-}NN^* 
\end{bmatrix}P^*.$$
\end{corollary}
\begin{proof}
By the definition of a GD-star matrix, we have 
\begin{align*}
    A^{GD,*}&=A^{GD}AA^*\\
    &=P\begin{bmatrix}
{C}^{-1}&O\\O&N^{-}\end{bmatrix}P^{*}P\begin{bmatrix}
{C}&O\\O&N\end{bmatrix}P^{*}
P\begin{bmatrix}
{C^*}&O\\O&N^*\end{bmatrix}P^*\\
&=P\begin{bmatrix}C^{-1}CC^*&O\\O&N^{-}NN^*
\end{bmatrix}P^*\\
&=P\begin{bmatrix}C^*&O\\O&N^{-}NN^* 
\end{bmatrix}P^*.
\end{align*}
\end{proof}
\begin{remark}
    Let $A\in\mathbb{C}^{m\times m}$ be an EP matrix. Then, \begin{align*}
    A=P\begin{bmatrix}
{C}&0\\0&0\end{bmatrix}P^{*},
\end{align*}
 where $P$ is unitary. So, \begin{align*}
    A^{GD}=P\begin{bmatrix}
{C}^{-1}&O\\O&M\end{bmatrix}P^{*}, 
\end{align*}
 where $M$ is an arbitrary matrix of suitable size, and 
$$A^{GD,*}=P\begin{bmatrix}C^*&O\\O&O \end{bmatrix}P^*.$$
\end{remark}

The {\it Hartwig-Spindelb$\ddot{\text{o}}$ck
decomposition} \cite{Har1,Har2} of any matrix $A\in\mathbb{ C}^{m\times m}$ of rank $r$ is given by
\begin{align}\label{dr1}
A=U\begin{bmatrix}
{\sum}K&{\sum}L\\O&O\end{bmatrix}U^*,\end{align}
where $U\in\mathbb{C}^{m\times m}$ is a unitary matrix and $\sum=diag(\sigma_1I_{r_1},\sigma_2I_{r_2},...,\sigma_tI_{r_t})$  is a diagonal matrix of the nonzero singular values of $A$, $\sigma_1>\sigma_2> ...>\sigma_t$, $r_1+r_2+...+r_t=r$
$K\in\mathbb{C}^{r\times r}$ and $L\in \mathbb{C}^{r\times {n-r}}$  satisfy
$$KK^*+LL^*=I_r.$$
The Hartwig-Spindelb$\ddot{\text{o}}$ck decomposition is useful to obtain characterizations of various known classes
of matrices. This decomposition is easily applicable while dealing with generalized and hypergeneralized
projectors.  Additionally, this representation is also used  to investigate properties of singular periodic matrices. 
Hernández {\it et al.} \cite{thome3} proposed the canonical form of a GD inverse using the Hartwig-Spindelb$\ddot{\text{o}}$ck decomposition. The same is recalled next.
\begin{theorem}(\cite{thome3})\label{iit}
    Let $A\in\mathbb{C}^{m\times m}$ be the form of \eqref{dr1} and $ind(A)=k$. Then,
a GD inverse of matrix $A$ is
$A^{GD}=U\begin{bmatrix}X_1&X_2\\X_3&X_4
\end{bmatrix}U^*$, where $X_i$ satisfies the following conditions:
\begin{enumerate}[(a)]
    \item $\sum KX_1+\sum L X_3=I_{r}$,
    \item $X_1(\sum K)^k=(\sum K)^{k-1}$,
    \item $X_3(\sum K)^{k-1}=O$,
    \item $(\sum K)^{k+1}X_2+(\sum K)^{k}X_4=(\sum K)^{k-1}\sum L$.
\end{enumerate}
\end{theorem}

Now, we give a representation of a GD-star matrix by using the Hartwig-Spindelb$\ddot{\text{o}}$ck decomposition.
\begin{theorem}
Let $A\in\mathbb{C}^{m\times m}$ be of the form as in \eqref{dr1} with $ind(A)=k$. If  $A^{GD}\in A\{GD\}$ is of the form as in Theorem \ref{iit}, then 
$$A^{GD,*}=U\begin{bmatrix}X_1\sum\sum^*&O\\
X_3\sum\sum^*&O
\end{bmatrix}U^*.$$

\end{theorem}
\begin{proof}By \eqref{dr1} and Theorem \ref{iit}, we have an expression for a GD-star matrix that is computed as
\begin{align*}
A^{GD,*}=A^{GD}AA^*&=U\begin{bmatrix}X_1&X_2\\X_3&X_4
\end{bmatrix}U^*U\begin{bmatrix}
{\sum}K&{\sum}L\\O&O\end{bmatrix}U^*U\begin{bmatrix}
({\sum}K)^*&O\\({\sum}L)^*&O\end{bmatrix}U^*\\
&=U\begin{bmatrix}X_1&X_2\\X_3&X_4
\end{bmatrix}\begin{bmatrix}
\sum K (\sum K)^*+\sum L(\sum L)^*&O\\O&O\end{bmatrix}U^*\\
&=U\begin{bmatrix}X_1&X_2\\X_3&X_4
\end{bmatrix}\begin{bmatrix}
\sum KK^*\sum^*+\sum LL^*\sum ^*&O\\O&O\end{bmatrix}U^*\\
&=U\begin{bmatrix}X_1&X_2\\X_3&X_4
\end{bmatrix}\begin{bmatrix}
\sum (KK^*+ LL^*)\sum ^*&O\\O&O\end{bmatrix}U^*\\
&=U\begin{bmatrix}X_1&X_2\\X_3&X_4
\end{bmatrix}\begin{bmatrix}
\sum\sum^*&O\\O&O\end{bmatrix}U^*\\
&=U\begin{bmatrix}X_1\sum\sum^*&O\\
X_3\sum\sum^*&O
\end{bmatrix}U^*.
\end{align*}
\end{proof}
Partial isometry matrix is an important class of matrix that plays an important role in the study of contraction \cite{par2, Har2}.
If $A$ is partial isometry, then we have the following result. 
\begin{theorem}
Let $A\in\mathbb{C}^{m\times m}$ be a partial isometry. The following conditions are equivalent: 
\begin{enumerate}[(i)]
    \item $X(A^{\dagger})^*X=X$ and $AX=AA^*$,
    \item There exists $A^{GD}\in A\{GD\}$ such that $X=A^{GD}AA^*+(I_m-A^{GD}A)AA^*$,
    \item $X(A^{\dagger})^*X=X$ and $AX=AA^{\dagger}$.
\end{enumerate}
\end{theorem}
\begin{proof}
(i)$\iff$(ii) Let $A^{GD}\in A\{GD\}$. Now,
 \begin{align}\label{kumar1}
     AX&=A(A^{GD}AA^*+(I_m-A^{GD}A)AA^*)\nonumber\\
     &=AA^{GD}AA^*+(A-AA^{GD}A)AA^*\nonumber\\
     &=AA^*.
 \end{align}
 \begin{align}
     X(A^{\dagger})^*X&=X(A^{\dagger}AA^{\dagger})^*X\nonumber\\
     &=X(A^{\dagger})^*A^{\dagger}AX\nonumber\\
     &=X(A^{\dagger})^*A^{\dagger}AA^*\nonumber\\
     &=(A^{GD}AA^*+(I_m-A^{GD}A)AA^*)(A^{\dagger})^*A^*\nonumber\\
     &=A^{GD}AA^*+(I_m-A^{GD}A)AA^*\nonumber\\
     &=X.
 \end{align}
 (ii)$\iff$(iii) The proof is analogous to the  above part.
 

\end{proof}

\begin{remark}
\begin{enumerate}
    \item If $A$ is both EP and partial isometry, then $A^{GD,*}=A^{\dagger}$.
    \item If $A$ is partial isometry, then  a GD-star matrix is also an $\{1,3\}$-inverse, so $A^{GD,*}b$ give a least-squares solution of system $Ax=b$.
\end{enumerate}
\end{remark}

We end this section with an application of a GD-star matrix.
 \begin{theorem}
 Let $A\in \mathbb{C}^{m\times m}$ with  $ind(A)=k$ and $b\in\mathbb{C}^{m}$. Then, the equation
 \begin{equation}\label{sys2.13}
     Ax=AA^*b
 \end{equation}
is consistent and its general solution is
$$x=A^{GD,*}b+(I_m-A^{GD}A)z,$$
where $z\in\mathbb{C}^{m\times 1}$ is arbitrary.
 \end{theorem}
\begin{proof}
Let $x=A^{GD}AA^*b$. Then, 
\begin{align*}
    Ax&=AA^{GD}AA^*b\nonumber\\
    &=AA^*b.
\end{align*}
Hence, the system \eqref{sys2.13} is consistence. Now,
$Ax=AA^{GD,*}b+A(I_m-A^{GD}A)z=AA^*b$. Also, if $x$ is a solution of \eqref{sys2.13}, then $x=A^{GD,*}b+(I_m-A^{GD}A)x$.  So, the general solution of the system \eqref{sys2.13} is given by $x=A^{GD,*}b+(I_m-A^{GD}A)z$. 
\end{proof}

The proof of the spectral result for a GD-star matrix is similar as the proof of Theorem \ref{spe}. So, we skip the proof.
\begin{theorem}
 Let $A\in \mathbb{C}^{m\times m}$ be a non-zero matrix and let $AA^*=\displaystyle\sum_{i=1}^r\alpha_iE_i$ be the spectral representation of $AA^*$, where $\{\alpha_1,\alpha_2,...,\alpha_r\}$ is the set of
distinct eigenvalues of $AA^*$ and $E_i$ are the corresponding self-adjoint spectral projections. Then, we have
$$A^{GD,*}=\displaystyle\sum_{i=1}^r\alpha_iA^{GD}E_i.$$
\end{theorem}

Let us denote $A^{GD,*}\in A\{GD\}AA^*$. We now obtain sufficient conditions so that $(AB)^{GD,*}=B^{GD,*}A^{GD,*}$.

\begin{theorem}
Let $A,B\in\mathbb{C}^{m \times m}$ be such that $AB=BA$ and $max\{ind(A),ind(B)\}=k$. If $R(A^{GD})\subseteq R(B)$ and $A^{GD}ABB^{*}=BB^{*}A^{GD}A$,  then $(AB)^{GD,*}=B^{GD,*}A^{GD,*}$.
\end{theorem}
\begin{proof} From $A^{GD}ABB^{*}=BB^{*}A^{GD}A$ and Theorem \ref{rev}, we get
\begin{align*}
    (AB)^{GD,*}&=(AB)^{GD}AB(AB)^*\\
    &=B^{GD}A^{GD}ABB^*A^*\\
    &=B^{GD}BB^*A^{GD}AA^*\\
    &=B^{GD,*}A^{GD,*}.
\end{align*}
\end{proof}
A few sufficient conditions for the forward-order law are provided below. 
\begin{theorem}
Let $A,B\in\mathbb{C}^{m \times m}$ be such that $AB=BA$ and $max\{ind(A),ind(B)\}=k$. If $R(B^{GD})\subseteq R(A)$ and $B^{GD}BAA^{*}=AA^*B^{GD}B$, then $(AB)^{GD,*}=A^{GD,*}B^{GD,*}$.
\end{theorem}
\begin{proof} From $B^{GD}BAA^{*}=AA^*B^{GD}B$ and Theorem \ref{for}, we have
\begin{align*}
    (AB)^{GD,*}&=(AB)^{GD}AB(AB)^*\\
    &=A^{GD}B^{GD}BAB^*A^*\\
    &=A^{GD}B^{GD}BAA^*B^*\\
    &=A^{GD}AA^*B^{GD}BB^*\\
    &=A^{GD,*}B^{GD,*}.
\end{align*}
\end{proof}

Kumar {\it et al.} \cite{k11} provided the additive properties of a GD inverse.
\begin{theorem}\label{theorem5.6}(\cite{k11})
Let $A,B\in\mathbb{C}^{m\times m}$ with $AB=BA=O$ and $max\{ind(A), ind(B)\}=k.$ If $A^{GD}B=BA^{GD}=O$ and $B^{GD}A=AB^{GD}=O$, then $(A+B)^{GD}=A^{GD}+B^{GD}.$ 
\end{theorem}
With the help of previous result, we describe the additive properties for a GD-star matrix. 
\begin{theorem}\label{rm5.6}
Let $A,B\in\mathbb{C}^{m\times m}$ with $AB=BA=BA^*=O$ and $max\{ind(A), ind(B)\}=k.$ If $A^{GD}B=BA^{GD}=O$ and $B^{GD}A=AB^{GD}=O$, then $(A+B)^{GD,*}=A^{GD,*}+B^{GD,*}.$ 
\end{theorem}
\begin{proof} Using the expression of a GD-star matrix.
\begin{align*}
    (A+B)^{GD}(A+B)(A+B)^*&=(A^{GD}+B^{GD})(AA^*+AB^*+BA^*+BB^*)\\
    &=A^{GD}AA^*+A^{GD}AB^*+A^{GD}BA^*+A^{GD}BB^*+B^{GD}AA^*\\
    &+B^{GD}AB^*+B^{GD}BA^*+B^{GD}BB^*\\
    &=A^{GD,*}+A^{GD}AB^*+B^{GD}BA^*+B^{GD,*}\\
        &=A^{GD,*}+B^{GD,*}.~~~~ 
\end{align*}
\end{proof}

\begin{theorem}\label{rm5.6}
Let $A,B\in\mathbb{C}^{m\times m}$ with $max\{ind(A), ind(B)\}=k.$ If $A^{GD,*}((A^*)^{\dagger}+(B^{*})^{\dagger})B^{GD,*}=A^{GD,*}+B^{GD,*},$ then  
\begin{enumerate}[(i)]
  \item $AA^*BB^{\dagger}=AA^*$.
    \item 
    $A^{GD}AB^{GD}B=B^{GD}B$. 
\end{enumerate}
\end{theorem}
\begin{proof}
    \begin{itemize}[(i)]
        \item  We have 
        \begin{align}\label{abs1}
            A^{GD,*}((A^*)^{\dagger}+(B^{*})^{\dagger})B^{GD,*}=A^{GD,*}+B^{GD,*}.
        \end{align}
        Pre-multiplying \eqref{abs1} by $A$, we get $AA^{GD,*}((A^*)^{\dagger}+(B^{*})^{\dagger})B^{GD,*}=AA^{GD,*}+AB^{GD,*}$, i.e., $AA^*((A^*)^{\dagger}+(B^{*})^{\dagger})B^{GD,*}=AA^*+AB^{GD,*}$, i.e, $AA^*(A^*)^{\dagger}B^{GD,*}+AA^*(B^{*})^{\dagger}B^{GD,*}=AA^*+AB^{GD,*}$, i.e., $AB^{GD,*}+AA^*(B^*)^{\dagger}B^{\dagger}BB^{GD}BB^*=AA^*+AB^{GD,*}$, i.e., $AA^*BB^{\dagger}=AA^*$. Hence, $AA^*BB^{\dagger}=AA^*$.
        \item[(ii)] 
        Now, post-multiplying \eqref{abs1} by $(B^*)^{\dagger}$, we obtain $A^{GD,*}((A^*)^{\dagger}+(B^{*})^{\dagger})B^{GD,*}(B^*)^{\dagger}=A^{GD,*}(B^*)^{\dagger}+B^{GD,*}(B^*)^{\dagger}$, i.e., $A^{GD,*}((A^*)^{\dagger}+(B^{*})^{\dagger})B^{GD}B=A^{GD,*}(B^*)^{\dagger}+B^{GD}B$, i.e., $A^{GD}AB^{GD}B+A^{GD,*}(B^{\dagger})^*B^{\dagger}BB^{GD}B=A^{GD,*}(B^*)^{\dagger}+B^{GD}B$. Using $BB^{GD}B=B$, we get $A^{GD}AB^{GD}B+A^{GD,*}(B^{\dagger})^*B^{\dagger}B=A^{GD,*}(B^*)^{\dagger}+B^{GD}B$, i.e., $A^{GD}AB^{GD}B+A^{GD,*}(B^*)^{\dagger}=A^{GD,*}(B^*)^{\dagger}+B^{GD}B$, i.e.,  $A^{GD}AB^{GD}B=B^{GD}B$.

    \end{itemize}
\end{proof}

\subsection{GD-star order}\label{sec4}
This section presents a binary relation called a GD-star order and its properties. We will start with the definition of a GD-star order.
\begin{definition}
 Let $A, B \in\mathbb{C}^{m\times m}$. Then, $A$ is said to be below $B$ under a GD-star order if there exists a GD-star inverse $A^{GD,*}$ of 
 $A$ such that
 $$A^{GD,*}A=A^{GD,*}B, \text{ and } AA^{GD,*}=BA^{GD,*}.$$
 It is denoted by $A\leq_{GD}^*B$.
\end{definition}
Next, we describe a characterization of a GD-star order.
\begin{theorem}
Let $A\in\mathbb{C}^{m\times m}$ be such that the matrix $P$ is a unitary matrix in \eqref{GD1}. Then,  $A\leq_{GD}^*B$ if and only if 
$$B=P\begin{bmatrix}C^*&O\\O&B_4
\end{bmatrix}P^{*},$$
where $B_4N^{-}N=N$ and $N^*B_4=N^*N$.
\end{theorem}
\begin{proof}
Let $B=P\begin{bmatrix}
B_1&B_2\\B_3&B_4\end{bmatrix}P^{*}$. Suppose that $A\leq_{GD}^*B$. So,
\begin{align*}
    A^{GD,*}A-A^{GD,*}B=\begin{bmatrix}
    C^*C-C^*B_1&-C^*B_2\\-N^{-}NN^*B_3&N^{-}NN^*N-N^{-}NN^*B_4\end{bmatrix}=O
\end{align*}
and 
\begin{align*}
    AA^{GD,*}-BA^{GD,*}=\begin{bmatrix}
    CC^*-B_1C^*&-B_2N^{-}NN^*\\-B_3C^*&NN^{-}NN^*-B_4N^{-}NN^*\end{bmatrix}=O
\end{align*}
From the above two equations, we get $C=B_1$, $B_2=O$, $B_3=O$, $N^{-}NN^*N=N^{-}NN^*B_4$ and $NN^{-}NN^*=B_4N^{-}NN^*$. Pre-multiplying by $N^{\dagger}N$ in $N^{-}NN^*N=N^{-}NN^*B_4$, we have $N^*N=N^*B_4$. Post-multiplying by $(N^{\dagger})^*$ in  $NN^{-}NN^*=B_4N^{-}NN^*$, we get $N=B_4N^{-}N$. Conversely, from the equality $N^*B_4=N^*N$, we obtain $N^{GD,*}N=N^{GD,*}B$. Further, $N^{GD}\in N\{1\}$ and $B_4N^{-}N=N$ imply that $B_4N^{GD}N=N$. Post-multiplying by $N^*$, we get $B_4N^{GD}NN^*=NN^{GD}NN^*$, i.e., $B_4N^{GD,*}=NN^{GD,*}$. So, $N\leq_{GD}^*B_4$. Hence, $A\leq_{GD}^*B$. 
\end{proof}
We know that an $\{1,4\}$-inverse of a matrix $N$ is also an $\{1\}$-inverse of matrix $N$.
\begin{remark}
In equation \eqref{GD1}, replace  $N^{-}$ by $N^{(1,4)}$. Then, 
$A\leq_{GD}^*B$ if and only if 
$$B=P\begin{bmatrix}C^*&O\\O&B_4
\end{bmatrix}P^*,$$
where $N\leq^{*}B_4$.
\end{remark}
\begin{theorem}
Let $A,B,C\in\mathbb{C}^{m\times m}$. Then,
\begin{enumerate}[(i)]
    \item GD-star order is reflexive.
    \item If $A\leq_{GD}^*B$ and $B\leq_{GD}^*A$, then $A=B$, i.e., GD-star order is anti-symmetric.
    \item If $A\leq_{GD}^*B$ and $B\leq_{GD}^*C$, then $A^{GD,*}A=A^{GD,*}CB^{GD}B$ and $AA^{GD,*}=CB^{GD}BA^{GD,*}$.  
\end{enumerate}

\end{theorem}
\begin{proof} (i) It is obvious.\\
(ii) From $A\leq_{GD}^*B$, we get 
$A^{GD,*}A=A^{GD,*}B$ and $AA^{GD,*}=BA^{GD,*}$. By the last equality $AA^{GD,*}=BA^{GD,*}$, we have $AA^{GD}AA^*=BA^{GD}AA^*$, i.e., $AA^*=BA^{GD}AA^*$. Post-multiplying by $(A^{\dagger})^*$ in $AA^*=BA^{GD}AA^*$, we obtain $A=BA^{GD}A$. 
Again, by the same process $B\leq_{GD}^*A$ implies $B=AB^{GD}B$. 
Now, pre-multiplying  $A^{GD,*}A=A^{GD,*}B$ by $A$, we obtain $AA^*A=AA^*B$ implies $A^*A=A^*B$. Again, pre-multiplying $A^*A=A^*B$ by $(A^{\dagger})^*$, we obtain $A=AA^{\dagger}B$. Using $B=AB^{GD}B$, we have $A=AA^{\dagger}AB^{GD}B$, i.e., $A=AB^{GD}B$. So, $A=B$.
\\
(iii) From $A\leq_{GD}^*B$ and $B\leq_{GD}^*C$, we get 
$A^{GD,*}A=A^{GD,*}B$, $AA^{GD,*}=BA^{GD,*}$, $B^{GD,*}B=B^{GD,*}C$ and $BB^{GD,*}=CB^{GD,*}$. Now, 
\begin{align*}
    A^{GD,*}A&=A^{GD,*}B\\
    &=A^{GD,*}BB^{GD}B\\
    &=A^{GD,*}BB^{GD}BB^{\dagger}B\\
    &=A^{GD,*}BB^{GD}BB^*(B^{\dagger})^*\\
    &=A^{GD,*}BB^{GD,*}(B^{\dagger})^*\\
    &=A^{GD,*}CB^{GD,*}(B^{\dagger})^*\\
    &=A^{GD,*}CB^{GD}BB^*(B^{\dagger})^*\\
    &=A^{GD,*}CB^{GD}B.
\end{align*}
Similarly, $AA^{GD,*}=CB^{GD}BA^{GD,*}$.   
\end{proof}
Next result is an equivalence relation between a GD-star order and group partial order.

\begin{theorem}

Let $A,B\in\mathbb{C}^{m\times m}$ be such that $ind(A)=1$ and $AB=BA$. Then, the following are equivalent:
\begin{enumerate}[(i)]
    \item $A\leq_{GD}^*B$,
    \item $A^{GD}AA^*=A^{GD,*}BA^{\dagger}$ and $A=BA^{GD}A$,
    \item $AA^*=AA^*BA^{\dagger}$ and $A^{k+1}=BA^{k}$,
   \item $A^*=A^*BA^{\dagger}$ and $A^{2}=BA$,
   \item $R(A^*)=R(A^*B)$ and $A^{2}=BA$.
\end{enumerate}
\end{theorem}
\begin{proof} (i)$\Rightarrow $(ii) $A\leq_{GD}^*B$ implies that $A^{GD,*}A=A^{GD,*}B, \text{ and } AA^{GD,*}=BA^{GD,*}.$ Post-multiplying by $A^{\dagger}$ in $A^{GD,*}A=A^{GD,*}B$, we obtain $A^{GD}AA^*=A^{GD,*}BA^{\dagger}$.
Again, post-multiplying $AA^{GD,*}=BA^{GD,*}$ by $(A^{\dagger})^*$, we get $AA^{GD}AA^*(A^{\dagger})^*=BA^{GD}AA^*(A^{\dagger})^*$ implies that $AA^{\dagger}A=BA^{GD}AA^{\dagger}A$, i.e., $A=BA^{GD}A$.\\
(ii)$\Rightarrow$(iii) Pre-multiplying by $A$ in $A^{GD}AA^*=A^{GD,*}BA^{\dagger}$, we get $AA^*=AA^*BA^{\dagger}$. Post-multiplying $A=BA^{GD}A$ by $A^{k}$, we obtain $A^{k+1}=BA^{k}$.\\
(iii)$\Rightarrow$(iv) Pre-multiplying $AA^*=AA^*BA^{\dagger}$ by $A^{\dagger}$, we have $A^*=A^*BA^{\dagger}$. Post-multiply $A^{k+1}=BA^k$ by $(A^{\#})^kA$, we arrive  $A^2=BA$. \\
(iv)$\Rightarrow$(v) $R(A^*)=R(A^*BA^{\dagger})\subseteq R(A^*B)\subseteq R(A^*)$, i.e., $R(A^*)=R(A^*B)$.\\
(v)$\Rightarrow$(i) The index of $A$ is 1. So, $A^{GD}A^2=A$, and
 $A^{2}=BA$ implies that $AA^{GD}A^2=BA^{GD}A^2$. Post-multiplying $AA^{GD}A^2=BA^{GD}A^2$ by $A^\#A^*$, we get $AA^{GD}AA^*=BA^{GD}AA^*$, i.e., $AA^{GD,*}=BA^{GD,*}$. We have $AB=BA$ and $A^2=BA$. So,  $A^{2}=AB$. Similarly, we obtain $A^{GD,*}A=A^{GD,*}B$. Hence, $A\leq_{GD}^*B$.

\end{proof} For simplicity,
if $A^{D,\dagger}A=A^{D,\dagger}B$ and $AA^{D,\dagger}=BA^{D,\dagger}$, then we denote it $A\leq_{D}^{\dagger}B$.  Suppose $A$ is below $B$ under a GD-star order, then $A$ and $B$ satisfy some properties are stated below. 
\begin{theorem}
 Let $A, B \in\mathbb{C}^{m\times m}$ be such that $ind(A)=k$. If $A$ is  below $B$ under a GD-star order, then the following hold: 
 \begin{enumerate}[(i)]
     \item $A^*A=A^*B$.
     \item $A=AA^{\dagger}B$.
     \item $A^{k+1}=BA^k$.
     \item $AA^D=BA^D$.
     \item $A\leq_D^{\dagger}B$.
     \item $A^{GD,\dagger}A=A^{GD,\dagger}B$ and $BA^{GD,\dagger}=AA^{GD,\dagger}.$
 \end{enumerate}
\end{theorem}
\begin{proof}
We have  $A$ is  below $B$ under a GD-star order, i.e.,   $$A^{GD,*}A=A^{GD,*}B, \text{ and } AA^{GD,*}=BA^{GD,*}.$$
(i)
Pre-multiplying by $A$ in $A^{GD,*}A=A^{GD,*}B$, we get $AA^{GD}AA^*A=AA^{GD}AA^*B$, i.e., $AA^*A=AA^*B$. Again, pre-multiplying by $A^{\dagger}$ in $AA^*A=AA^*B$, we get $A^{\dagger}AA^*A=A^{\dagger}AA^*B$, i.e., $A^*(A^*)^{\dagger}A^*A=A^*(A^*)^{\dagger}A^*B$, i.e., $A^*A=A^*B$. \\
(ii) It is easily obtained  pre-multiplying (i) by $(A^{\dagger})^*$, we get $AA^{\dagger}A=AA^{\dagger}B$, i.e., $A=AA^{\dagger}B$.\\
(iii) Post-multiplying by $(A^{\dagger})^*$ in $AA^{GD,*}=BA^{GD,*}$, we get  $AA^{GD}AA^*(A^{\dagger})^*=BA^{GD}AA^*(A^{\dagger})^*$, i.e, $A(A^{\dagger}A)^*=BA^{GD}A(A^{\dagger}A)^*$, i.e., $A=BA^{GD}A$. Again, post-multiplying by $A^k$ in $A=BA^{GD}A$, we get $A^{k+1}=BA^k$.\\
(iv)
Every square matrix is Drazin invertible. So, post-multiplying   $A^{k+1}=BA^k$ by $(A^{D})^{k+1}$, we get $AA^{D}=BA^D$.\\
(v) Post-multiplying by $AA^{\dagger}$ in (iv), we get $AA^{D,\dagger}=BA^{D,\dagger}$. Pre-multiplying by $A^{\dagger}A$ in $A^{GD,*}A=A^{GD,*}B$, we get $A^*A=A^*B$. Again, pre-multiplying $A^*A=A^*B$ by $A^DAA^{\dagger}(A^{\dagger})^*$, we otain $A^{D,\dagger}A=A^{D,\dagger}B$. Hence, $A\leq_D^{\dagger}B$.\\
(vi) Pre-multiplying  $A^{GD,*}A=A^{GD,*}B$ by $(A^{\dagger})^*A$, we have $A^*A=A^*B$. Again, pre-multiplying by $A^{GD}AA^{\dagger}(A^{\dagger})^*$ in $A^*A=A^*B$, we obtain $A^{GD,\dagger}A=A^{GD,\dagger}B$. Post-multiplying   $AA^{GD,*}=BA^{GD,*}$ by $(A^{\dagger})^*A^{\dagger}$, we get $AA^{GD,\dagger}=BA^{GD,\dagger}$.
\end{proof}
If $A$ is below $B$ under the minus partial order and $*$ partial order, then $A$ is below $B$ under a GD-star order.
\begin{theorem}
 Let $A, B \in\mathbb{C}^{m\times m}$ be such that $ind(A)=k$. If $A\leq^{-}B$ and $A\leq^{*}B$, then $A\leq_{GD}^{*} B$.
 \end{theorem}
 \begin{proof}
 The hypothesis $A\leq^{-}B$ implies that
 \begin{equation}\label{pri1}
     AA^{-}=BA^{-},
 \end{equation}
 and
 \begin{equation}\label{pri2}
     A^{-}A=A^{-}B.
 \end{equation}
 We know that  $A^{GD}\in A\{1\}$. From \eqref{pri1} and $A^{GD}\in A\{1\}$, we obtain $AA^{GD,*}=AA^{GD}AA^*=BA^{GD}AA^*=BA^{GD,*}$, i.e., $AA^{GD,*}=BA^{GD,*}$. Similarly, from $A\leq^*B$,  we get $A^{GD,*}A=A^{GD,*}B$. Hence, $A\leq_{GD}^{*}B$.\\
 
 \end{proof}
 
Using similar step as in the above one can prove the next result.
 \begin{corollary}
 Let $A, B \in\mathbb{C}^{m\times m}$ be such that $ind(A)=k$. If $A\leq^{GD}B$ and $A\leq^*B$, then $A\leq_{GD}^{*} B$.
 \end{corollary}
From $A\leq^{GD}B$, we have $A^{GD}A=A^{GD}B$ and $AA^{GD}=BA^{GD}$. Pre-multiplying $A^{GD}A=A^{GD}B$ by $(A^D)^{k+1}A^{k+1}$, we obtain $(A^D)^{k+1}A^{k+1}=(A^D)^{k+1}A^kB$, i.e., $A^DA=A^DB$. Similarly, post-multiplying $AA^{GD}=BA^{GD}$ by $A^{k+1}(A^D)^{k+1}$, we get $AA^D=BA^D$. Therefore, $A\leq^DB$.  We conclude this section with the following remark.
 \begin{remark}

 If $A\leq^{GD}B$, then $A\leq^DB$.
 \end{remark}

\subsection{Dual GD-star or star-GD matrices}\label{sec6}
\noindent 
In this section, we discus about the dual of a GD-star matrix. The proofs of a dual GD-star matrix are similar to the proof of a GD-star matrix. Owing the similarity, we deal with a few important results of a dual GD-star matrix. We will start this section with the definition of a dual GD-star matrix.
\begin{definition}\label{dsa1}
Let $A \in \mathbb{C}^{m\times m}$ and $ ind(A)=k$. Let  $A^{GD} \in A\{GD\}$, a dual GD-star matrix of $A$, denoted by $A^{*,GD}$, be an $m \times m$ matrix
$$A^{*,GD} = A^{*}AA^{GD}.$$
\end{definition}
\begin{theorem}\label{dsa2}
Let $A\in \mathbb{C}^{m\times m}$. Then, the system 
$$X(A^{\dagger})^*X=X,~XA^k=A^*A^k, \text{ and } (A^{\dagger})^*X=AA^{GD}$$
has a solution of the form $X=A^{*}AA^{GD}$, for every nonnegative integer $k$. 
\end{theorem}
\begin{lemma}\label{dsa3}
Let $A \in \mathbb{C}^{m\times m}$, and $ind(A)=k$. If $A^{GD}\in A\{GD\}$, then a dual GD-star matrix $X$ of the matrix $A$ satisfies the following properties:
\begin{enumerate}[(i)]
   
     \item $(A^{\dagger})^*XA=A$.
    \item $XA^k=A^*A^{k+1}A^{GD}=A^*A^{GD}A^{k+1}$.
    \item $XAA^{\dagger}=A^*$.
    \item $(A^{\dagger})^*XA^k=A^k$.
    \item $A^{GD}(A^{\dagger})^*X=A^{GD}AA^{GD}$.
    \item $X(A^{\dagger})^*A^k=A^k$.
    \item $X^2AA^{\dagger}=XA^*$.
    \item $XAX=(A^*A)^2A^{GD}$.
    \item $XA^{\dagger}AX=X^2$.
    \item $(XA)^*=XA$.
    \item $(A^{\dagger})^*X(A^{\dagger})^*=(A^{\dagger})^*.$
    \item $X(A^{\dagger})^*=A^{\dagger}A.$
\end{enumerate}
\end{lemma}
\begin{theorem}
Let $A\in\mathbb{C}^{m\times m}$ be the form of \eqref{ka5} and $ind(A)=k$. If  $A^{GD}\in A\{GD\}$, then 
\begin{align}\label{ka6} A^{*,GD}=P\begin{bmatrix}C^*&C^*C^{-k}(\widehat{T'}-\widehat{T}N^{-})+C^*SN^{-}\\S^*&S^*C^{-k}(\widehat{T'}-\widehat{T}N^{-})+N^*NN^{-}\end{bmatrix}P^*,\end{align}
where $P$ is a unitary matrix and $\widehat{T'}-\widehat{T}N^{-}=-C^{k}SN^{-}+\displaystyle\sum_{j=0}^{k-1}C^{j}SN^{k-j-1}(I-NN^{-})$.
\end{theorem}

\subsection{GD-star-one matrices}\label{sec7}
In this subsection, we define a new class of matrices called GD $*_1$ and investigate a few of its properties.   First we define GD $*_1$ matrices as follows.   
\begin{definition}\label{dsa1}
Let $A \in \mathbb{C}^{m\times m}$ and $ ind(A)=k$. For  $A^{GD} \in A\{GD\}$, a  GD $*_1$ matrix of $A$, denoted by $A^{GD,*_1}$, is an $m \times m$ matrix
$$A^{GD,*_1} = A^{GD}A^{*}A.$$
\end{definition}
An example that shows a GD-star matrix is different from a GD $*_1$ matrix. Obviously, a GD $*_1$ of a matrix $A \in \mathbb{C}^{m\times m}$  is also not unique.
\begin{example}
From Example \ref{exam3}, $A^{GD}=\begin{bmatrix}1&a\\0&b\end{bmatrix}$ is a GD inverse of $A=\begin{bmatrix}1&1\\0&0\end{bmatrix}$. Now, $A^{GD,*_1}=A^{GD}A^*A=\begin{bmatrix}1&a\\0&b\end{bmatrix}\begin{bmatrix}1&0\\1&0\end{bmatrix}\begin{bmatrix}1&1\\0&0\end{bmatrix}=\begin{bmatrix}1+a&1+a\\b&b\end{bmatrix}$, $A^{GD,*}=A^{GD}AA^*=\begin{bmatrix}
    2&0\\0&0
\end{bmatrix}$ and $A^*AA^{GD}=\begin{bmatrix}1&a+b\\1&a+b\end{bmatrix}$. It is clear $A^{*,GD}\neq A^{GD,*_1}\neq A^{GD,*}.$
\end{example} 
Some properties of a GD $*_1$ inverse are given below. 
\begin{lemma}\label{dsa}
Let $A \in \mathbb{C}^{m\times m}$, and $ind(A)=k$. If $A^{GD}\in A\{GD\}$, then  a GD $*_1$ inverse $X$ of the matrix $A$ satisfies the following properties:
\begin{enumerate}[(i)]
\item $A^{k+1}X=A^kA^*A$.
\item $XA^{\dagger}=A^{GD}A^*$.
\item $A^{k+1}XA^{\dagger}=A^kA^*$.
\item If $A$ is EP, then $AXA^{\dagger}=A^*$.
\item $XA^{GD}A=X$.
\item $XY=A^{GD}A^*$, where $Y\in A\{GD,\dagger\}$.
\item $XZ=A^{GD}A^*A^{GD}$, where $Z\in A\{\dagger,GD\}$.
\item $A^{k+1}XY=A^kA^*$ where $Y\in A\{GD,\dagger\}$.
\end{enumerate}
\end{lemma}
\begin{proof}
(i) $A^{k+1}X=A^{k+1}A^{GD}A^*A=A^kA^*A$.\\
(ii) $XA^{\dagger}=A^{GD}A^*AA^{\dagger}=A^{GD}A^*(A^*)^{\dagger}A^*=A^{GD}A^*$.\\
(iii) $A^{k+1}XA^{\dagger}=A^{k+1}A^{GD}A^*AA^{\dagger}=A^kA^*(A^*)^{\dagger}A^*=A^kA^*$.\\
(iv) We have $AA^{\dagger}=A^{\dagger}A$. Then, $AXA^{\dagger}=AA^{GD}A^*AA^{\dagger}=AA^{GD}A^{\dagger}AA^*AA^{\dagger}=AA^{GD}AA^{\dagger}A^*AA^{\dagger}=AA^{\dagger}A^*=A^{\dagger}AA^*=A^*$.\\
(v) It is obvious.\\
(vi) Let $Y=A^{GD}AA^{\dagger}$. Then, $XY=A^{GD}A^*AA^{GD}AA^{\dagger}=A^{GD}A^*AA^{\dagger}=A^{GD}A^*$.\\
(vii) Similar to the above part.\\
(viii) From (vi), we obtain $XY=A^{GD}A^*$. Now, $A^{k+1}XY=A^{k+1}A^{GD}A^*=A^{k}A^*.$ 

\end{proof}

We conclude this section with an open problem.

{\bf Problem:} Let $A\in \mathbb{C}^{m\times m}$ with  $ind(A)=k> 1$. Then, consider the matrix equations
\begin{align*}
    AXA&=A\\
    XAX&=X\\
    A^{k+1}X&=XA^{k+1}=A^k.
\end{align*}
Under what conditions a solution of the above matrix equation exist? How does the solution look like?

\section{Conclusions}
    The notion of a GD-star matrix and its representation  for a square matrix has been introduced. Some properties of GD-star order have been presented. Some sufficient conditions are obtained so that the triple reverse and forward-order laws for GD and GD-star generalized inverses hold.
    The discussed results are useful for computation of absorption law and to solve a linear system. These theories can also be studied in a ring with involution and in a tensor setting.
One may also look for an integral representation of GD inverse and GD-star matrix in future research work.

 \noindent 





\section{Acknowledgements}
The first author acknowledges the support of the Council of Scientific and Industrial Research-University Grants Commission, India.





\begin{thebibliography}{10}

 \bibitem{Dilan}{D. Ahmed, M. Hama, K.H.F. Jwamer  and S. Shateyi.} {A seventh-Order scheme for computing the generalized Drazin inverse}. {\it Mathematics,}  7:622, 2019.
 \bibitem{Har2}{O.M. Baksalary, G.P.H. Styan, G. Trenkler. On a matrix decomposition of Hartwig and Spindelböck. {\it Linear Algebra Appl.}, 430:2798-2812, 2009.}


\bibitem{baks}{ O.M. Baksalary and G. Trenkler.} {Core inverse of matrices}. {\it Linear Multilinear Algebra,} 58(6):681-697, 2010.


\bibitem{app-1} J.K. Baksalary, J. Hauke and G.P.H. Styan. {On some distributional properties of quadratic
forms in normal variables and on some associated matrix partial orderings}. {\it Lecture Notes-Monograph Series}, 24:111-121, 1994.
\bibitem{app-2} J.K. Baksalary and S. Puntanen. {Characterizations of the best linear unbiased estimator in
the general Gauss-Markov model with the use of matrix partial orderings.} {\it Linear Algebra Appl.}, 127:363-370, 1990.
\bibitem{ben11}{A. Ben-Israel and T.N.E. Greville.} {Generalized Inverses: Theory and Applications}. {\it 2nd ed. New York: Springer-Verlag,} 2003.
\bibitem{rpar}{G.W. Brumfiel. Partially ordered rings and semi-algebraic geometry. {\it Cambridge University Press}, 1979.}
\bibitem{Barata}{J.C.A. Barata and  M.S. Hussein. The Moore–Penrose pseudoinverse: A tutorial review of the theory. {\it Braz. J. Phys.}, 42(1):146-165, (2012).}
\bibitem {CD}{S.L. Campbell and C.D. Meyer.} {Generalized Inverses of Linear Transformations}. {\it SIAM,} 2009.
\bibitem{sl1}{S.L. Campbell  and C.D. Meyer.}  Weak Drazin inverses. {\it Linear Algebra Appl.,} 20:167-178, 1978

\bibitem{Nieves}{N. {Castro-Gonz\'alez} and R.E. Hartwig.} {Perturbation results and the forward-order law for the Moore-Penrose inverse of a product}. {\it Electron. J. Linear Algebra,} 34:514-525, 2018.


\bibitem{camb1}{S.L. Campbell. Linear systems of differential equations with singular coefficients.}
{\it SIAM J. Appl. Math.,} 8:1057-1066, 1977.
\bibitem{camb2}{S.L. Campbell. Limit behavior of solutions of singular difference equations.} {\it Linear Algebra Appl.,} 23:167-178, 1979.
\bibitem{thome1}{C. Coll, M.B. Lattanzi and N. Thome.} {Weighted G-Drazin inverses and a new pre-order on rectangular matrices}. {\it Appl. Math. Comput.,}  317:12-24,  2018.

\bibitem{Deng24}{C.Y. Deng.} {Reverse-order law for the group inverses}. {\it J. Math. Anal. Appl.,}  382:663-671,  2011.

\bibitem{m1}{M.P. Drazin. {Natural structures on semigroups with involution.} {\it Bull. Amer. Math. Soc.}, 84(1):139-141, 1978.}

\bibitem{m2} {R.E. Hartwig. How to partially order regular elements. {\it Math. Japon.}, 25(1):1-13, 1980.}

\bibitem{par2}{P.R. Halmos and J.E. McLaughlin. Partial isometries. {\it Pacific J. Math.} 13:585-596, 1963.}

\bibitem{thes}{V.K. Gupta and B. Jayaram.} {Clifford’s order from fuzzy logic Connectives. {\it Diss. IIT Hyderabad,} 2022.
}
\bibitem{Guo}{W. Guo, M. Wei and Z. Jianli.} { Forward order law for g-inverses of the product of two matrices}. {\it Appl. Math. Comput.,} 189:1749-1754, 2007.
\bibitem{Grev}{T.N.E. Greville.} {Note on the generalized inverse of a matrix product}. {\it SIAM Rev.,} 8:518-521, 1966.
\bibitem {Har1}{R.E. Hartwig, K. Spindleböck. Matrices for which $A^*$ and $A^{\dagger}$ commute. {\it Linear Multilinear Algebra}, 14:241-256, 1984.}
\bibitem{Har2}{R.E. Hartwig  and K. Spindleböck. Partial isometries, contractions and EP matrices.} {\it Linear Multilinear Algebra,} 13:295-310, 1983.
\bibitem{thome3}{M.V. Hernández, M.B. Lattanzi and N.  Thome.} {GDMP-inverses of a matrix and their duals}. {\it Linear Multilinear Algebra,} 2020, DOI: 10.1080/03081087.2020.1857678.

\bibitem{app2} {{A. Hern\'andez, M.B. Lattanzi and N. Thome.} {On some new pre-orders defined by weighted Drazin inverses.} {\it Appl. Math. Comput.,} 282:108-116, 2016.}

\bibitem{app4}{{A. Hernández, M.B. Lattanzi, N. Thome and F. Urquiza. {The star partial order and the eigenprojection at 0 on EP matrices}. {\it Appl. Math. Comput.}, 218:10669-10678, 2012.}}
\bibitem{k11}{A. Kumar, V. Shekhar and D. Mishra.} {$W$-weighted GDMP inverse for rectangular matrices}. {\it Electron J. Linear Algebra,} 38:632-657, 2022. 

{\bibitem{app1}{{I.I. Kyrchei.} {Determinantal representations of the W-weighted Drazin inverse over the quaternion skew field}. {\it Appl. Math. Comput.}, 264:453-465, 2015.}}

\bibitem{Xiong25}{Z. Liu and Z. Xiong.}
{The forward-order laws for \{1, 2, 3\}- and \{1, 2, 4\}-inverses of a three matrix products}. {\it Filomat,}  32(2):589-598, 2018.





\bibitem{CDMeyer}{C.D.Jr. Meyer.} {The condition of a finite Markov chain and perturbation bounds for the limiting probabilities}. {\it SIAM J. Algebr. Discrete Methods,} 1:273-283,  1980.


\bibitem{m3}{S.K. Mitra. On group inverses and the sharp order. {\it Linear Algebra Appl.}, 92:17-37, 1987.}
\bibitem{Mitra}{S.K. Mitra, P. Bhimasankaram and S.B. Malik.} {Matrix partial orders, shorted operators and applications}. {\it World Scientific Publishing Company,} 2010.

\bibitem {Mosic110}{D. Mosi\'c.} {Drazin-star and star-Drazin matrices}. {\it Results Math.,}  75:1-21, 2020.

\bibitem{Mosic4}{D. Mosi\'c, I.I. Kyrchei and P.S. Stanimirovi\'c.} {Representations and properties for the MPCEP inverse}. {\it J. Appl. Math. Comput.,}  67:101-130,  2021.

\bibitem{Penrose}{R. Penrose}. {A generalized inverse for matrices}. {\it Cambridge Philosoph. Soc.,}  51:406-413, 1955.

\bibitem{stewart}{G.W. Stewart. On the continuity of the generalized inverse.} {\it SIAM J. Appl. Math.}, 17:33-45,  1969.
\bibitem{app3}{C. Stepniak. Ordering of nonnegative definite matrices with application to comparison of linear models. {\it Linear Algebra Appl.}, 70:67-71, 1987.}

\bibitem{Wang2} {H. Wang and X. Liu.} {Partial orders based on core-nilpotent decomposition}. {\it Linear Algebra Appl.,}  488:235-248,  2016.

\bibitem{Wang23}{X. Wang, A. Yu, T. Li and C. Deng.} {Reverse-order laws for the Drazin inverses}. {\it J. Math. Anal. Appl.,}  444:672-689,  2016.
 
\bibitem{Xiong16}{Z. Xiong and B. Zheng.} {The reverse-order laws for \{1, 2, 3\}- and \{1, 2, 4\}-inverses of a two-matrix product}. {\it 	
Appl. Math. Lett.,}  21:649-655, 2008.
\bibitem{Xion}{Z. Xiong and B. Zheng. Forward order law for the generalized inverses of multiple matrix product.} {\it J. Appl. Math.
Comput.,} 25:415-424, 2007.

 






\end{thebibliography}
\end{document}